\documentclass{amsart}
\usepackage{latexsym,graphics,euscript,amscd,amssymb,upgreek}
\usepackage[a4paper]{geometry}
\newcommand\cA{\EuScript{A}}
\newcommand\cG{\EuScript{G}}
\newcommand\cM{\EuScript{M}}
\newcommand\cN{\EuScript{N}}

\newcommand\bC{\mathbb{C}}
\newcommand\bR{\mathbb{R}}
\newcommand\bZ{\mathbb{Z}}

\newcommand\Lie{\mathrm{Lie}}
\newcommand\ad{\operatorname{ad}}
\newcommand\Ad{\operatorname{Ad}}
\newcommand\Hol{\mathrm{Hol}}
\newcommand\Hom{\operatorname{Hom}}

\newcommand\Aut{\operatorname{Aut}}
\newcommand\Inn{\operatorname{Inn}}
\newcommand\irr{\mathrm{irr}}
\newcommand\red{\mathrm{red}}
\newcommand\good{\mathrm{good}}

\newcommand\fk{\mathfrak{k}}

\newcommand\tM{\tilde{M}}
\newcommand\tP{\tilde{P}}
\newcommand\tphi{{\tilde\phi}}
\newcommand\tGam{{\tilde\Gam}}
\newcommand\tSig{{\tilde\Sig}}
\newcommand\al\alpha
\newcommand\gam\gamma
\newcommand\vph\varphi
\newcommand\sig\sigma
\newcommand\om\omega
\newcommand\upom\upomega
\newcommand\Gam\varGamma
\newcommand\Lam\varLambda
\newcommand\Ps\varPsi
\newcommand\Sig\varSigma
\newcommand\Om\varOmega

\newcommand\doubleslash{/ \negthinspace\negthinspace /}
\newcommand\doubleslashz{\doubleslash_{\!0\;\!}}
\newcommand\tripleslash{\doubleslash\negthinspace\negthinspace /_{\!0\;\!}}

\DeclareMathOperator\im{im}
\DeclareMathOperator\id{id}

\numberwithin{equation}{section}

\newtheorem{thm}{Theorem}[section]
\newtheorem{lem}[thm]{Lemma}
\newtheorem{cor}[thm]{Corollary}
\newtheorem{pro}[thm]{Proposition}
\theoremstyle{remark}
\newtheorem{rem}[thm]{Remark}

\newcommand\ii{\sqrt{-1}}
\newcommand\bra{\langle}
\newcommand\ket{\rangle}

\newcommand\dR{\mathrm{dR}}
\newcommand\sst{\mathrm{sst}}
\newcommand\pst{\mathrm{pst}}
\newcommand\bet{\mathrm{Betti}}
\newcommand\fla{\mathrm{flat}}
\newcommand\hit{\mathrm{Hitchin}}
\newcommand\hig{\mathrm{Higgs}}
\newcommand\Dol{\mathrm{Dol}}
\newcommand\HK{\mathrm{HK}}

\begin{document}

\parskip=0.3\baselineskip
\baselineskip=1.2\baselineskip

\vspace*{-1.25cm}

%\begin{flushright}
%{\tt arXiv:1211.0746 [math.DG]}\\
%Revised: September (2016)
%\end{flushright}

\vspace*{1.5cm}

\title{Hitchin's equations on a nonorientable manifold}
%\title[Hitchin's Equations on a Nonorientable Manifold]{Hitchin's equations on a nonorientable manifold}

\author{Nan-Kuo Ho}
\address{Department of Mathematics, National Tsing Hua University, Hsinchu 300, Taiwan and
National Center for Theoretical Sciences, Taipei 106,  Taiwan}
\email{nankuo@math.nthu.edu.tw}

\author{Graeme Wilkin}
\address{Department of Mathematics,
National University of Singapore, Singapore 119076}
\email{graeme@nus.edu.sg}

\author{Siye Wu}
\address{Department of Mathematics,
University of Hong Kong, Hong Kong \\ 
Current address: Department of Mathematics, National Tsing Hua University, Hsinchu 300, Taiwan}
\email{swu@math.nthu.edu.tw}

\keywords{moduli spaces, non-orientable manifolds,
symplectic and hyper-K\"ahler geometry, representation varieties}

\begin{abstract}
We define Hitchin's moduli space $\cM^\hit(P)$ for a principal bundle $P$,
whose structure group is a compact semisimple Lie group $K$, over a compact
non-orientable Riemannian manifold $M$.
We use the Donaldson-Corlette correspondence, which identifies Hitchin's
moduli space with the moduli space of flat $K^\bC$-connections, which
remains valid when $M$ is non-orientable.
This enables us to study Hitchin's moduli space both by gauge theoretical
methods and algebraically by using representation varieties.
If the orientable double cover $\tM$ of $M$ is a K\"ahler manifold with odd
complex dimension and if the K\"ahler form is odd under the non-trivial deck
transformation $\tau$ on $\tM$, Hitchin's moduli space $\cM^\hit(\tP)$ of
the pull-back bundle $\tP\to\tM$ has a hyper-K\"ahler structure and admits
an involution induced by $\tau$.
The fixed-point set $\cM^\hit(\tP)^\tau$ is symplectic or Lagrangian with
respect to various symplectic structures on $\cM^\hit(\tP)$.
We show that there is a local diffeomorphism from $\cM^\hit(P)$ to
$\cM^\hit(\tP)^\tau$.
We compare the gauge theoretical constructions with the algebraic approach
using representation varieties.\\

\noindent
2010 Mathematics Subject Classification: 53D30, 58D27
\end{abstract}

\maketitle

\section{Introduction}

Let $M$ be a compact orientable Riemannian manifold and let $K$ be a connected
compact Lie group.
Given a principal $K$-bundle $P\to M$, let $\cA(P)$ be the space of connections
and let $\cG(P)$ be the group of gauge transformations on $P$.
Consider Hitchin's equations
\begin{equation}\label{H-eq}
F_A-\tfrac12[\psi,\psi]=0,\quad d_A\psi=0,\quad d^*_A\psi=0
\end{equation}
on the pairs $(A,\psi)\in\cA(P)\times\Om^1(M,\ad P)$.
Hitchin's moduli space $\cM^\hit(P)$ is the set of space of solutions
$(A,\psi)$ to \eqref{H-eq} modulo $\cG(P)$ \cite{Hi,S88}.
On the other hand, let $G=K^\bC$ be the complexification of $K$ and let
$P^\bC=P\times_KG$, which is a principal bundle with structure group $G$.
The moduli space $\cM^\dR(P^\bC)$ of flat $G$-connections on $P^\bC$, also
known as the de Rham moduli space, is the space of flat reductive connections
of $P^\bC$ modulo $\cG(P)^\bC\cong\cG(P^\bC)$.
A theorem of Donaldson \cite{Do} and Corlette \cite{C} states that the
moduli spaces $\cM^\hit(P)$ and $\cM^\dR(P^\bC)$ are homeomorphic.
The smooth part of $\cM^\hit(P)$ is a K\"ahler manifold with a complex
structure $\bar J$ induced by that on $G$.

Suppose in addition that $M$ is a K\"ahler manifold.
Then there is another complex structure $\bar I$ on $\cM^\hit(P)$ induced by
that on $M$, and a third one given by $\bar K=\bar I\bar J$.
The three complex structures $\bar I,\bar J,\bar K$ and their corresponding
K\"ahler forms $\bar\upom_I,\bar\upom_J,\bar\upom_K$ form a hyper-K\"ahler
structure on (the smooth part of) $\cM^\hit(P)$ \cite{Hi,S88}.
This hyper-K\"ahler structure comes from an infinite dimensional version of a
hyper-K\"ahler quotient \cite{HKLR} of the tangent bundle $T\cA(P)$, which is
hyper-K\"ahler, by the action of $\cG(P)$, which is Hamiltonian with respect
to each of the K\"ahler forms $\upom_I,\upom_J,\upom_K$ on $T\cA(P)$.
When $M$ is a compact orientable surface, Hitchin's moduli space $\cM^\hit(P)$
is equal to the hyper-K\"ahler quotient $\cM^\HK(P):=T\cA(P)\tripleslash\cG(P)$
\cite{Hi}.
It plays an important role in mirror symmetry and geometric Langlands program
\cite{HT,KW}.
When $M$ is higher dimensional, $\cM^\hit(P)$ is a hyper-K\"ahler subspace
in $\cM^\HK(P)$ \cite{S88}.

For a compact Lie group $K$, the moduli space of flat $K$-connections on
a compact orientable surface was already studied in a celebrated work of
Atiyah and Bott \cite{AB}.
When $M$ is a compact, nonorientable surface, the moduli space of flat
$K$-connections was studied in \cite{Ho,HL2} through an involution on the
space of connections over its orientable double cover $\tM$, induced by
lifting the deck transformation on $\tM$ to the pull-back $\tP\to\tM$ of
the given $K$-bundle $P\to M$ so that the quotient of $\tP$ by the
involution is the original bundle $P$ itself.
This involution acts trivially on the structure group $K$.
If instead one considers an involution on the bundle over $\tM$ that acts
nontrivially on the fibers (such as the complex conjugation), then the fixed
points give rise to the moduli space of real or quaternionic vector bundles
over a real algebraic curve.
This was studied thoroughly in \cite{BHH,Sch}, for example when $K=U(n)$.

In this paper, we study Hitchin's equations on a non-orientable manifold.
Let $M$ be a compact connected non-orientable Riemannian manifold and
let $P\to M$ be a principal $K$-bundle over $M$, where $K$ is a compact
connected Lie group.
The de Rham moduli space $\cM^\dR(P^\bC)$, i.e., the moduli space of flat
connections on $P^\bC$, does not depend on the orientability of $M$.
On the other hand, Hitchin's equations \eqref{H-eq} on the pairs
$(A,\psi)\in\cA(P)\times\Om^1(M,\ad P)$ still make sense (see
subsection~\ref{sec:DC}).
We define Hitchin's moduli space $\cM^\hit(P)$ as the quotient of the space
of pairs $(A,\psi)$ satisfying \eqref{H-eq} by the group $\cG(P)$ of gauge
transformations on $P$.
We explain that the homeomorphism $\cM^\hit(P)\cong\cM^\dR(P^\bC)$
of Donaldon-Corlette remains valid when $M$ is non-orientable
(Theorem~\ref{thm:DC}).

If the oriented cover $\tM$ of $M$ is a K\"ahler manifold, then for the
pull-back bundle $\tP:=\pi^*P$ over $\tM$, Hitchin's moduli space
$\cM^\hit(\tP)$ is hyper-K\"ahler with complex structures
$\bar I,\bar J,\bar K$ and K\"ahler forms
$\bar\upom_I,\bar\upom_J,\bar\upom_K$.
If the K\"ahler form $\om$ on $\tM$ satisfies $\tau^*\om=-\om$ (the complex
dimension of $\tM$ must be odd for $\tau$ to be orientation reversing), then
$\tau$ induces an involution (still denoted by $\tau$) on $\cM^\hit(\tP)$ that
satisfies $\tau^*\bar\upom_I=-\bar\upom_I$, $\tau^*\bar\upom_J=\bar\upom_J$
and $\tau^*\bar\upom_K=-\bar\upom_K$.
Consequently, the fixed-point set $(\cM^\hit(\tP))^\tau$ is Lagrangian in
$\cM^\hit(\tP)$ with respect to $\bar\upom_I,\bar\upom_K$ and symplectic with
respect to $\bar\upom_J$.
This is known as an (A,B,A)-brane in \cite{KW}.
We discover that Hitchin's moduli space $\cM^\hit(P)$ (where $M$ is
non-orientable) is related to $(\cM^\hit(\tP))^\tau$ by a local diffeomorphism.
Our main results are summarized in the following main theorem.
For simplicity, we restrict to certain smooth parts $\cM^\hit(P)^\circ$,
$\cM^\hit(\tP)^\circ$ and $\cA^\fla(P^\bC)^\circ$ of the respective spaces
(see subsection~\ref{sec:Hitchin} for details).

\begin{thm}\label{thm:Hitchin}
Let $M$ be a compact non-orientable manifold and let $\pi\colon\tM\to M$ be
its oriented cover on which there is a non-trivial deck transformation $\tau$.
Let $K$ be a compact connected Lie group.
Given a principal $K$-bundle $P\to M$, let $\tP=\pi^*P$ be its pull-back to
$\tM$.
Suppose that $\tM$ is a K\"ahler manifold of odd complex dimension and
the K\"ahler form $\om$ on $\tM$ satisfies $\tau^*\om=-\om$.
Then\\
(1) $\cM^\hit(P)^\circ=\cA^\fla(P^\bC)^\circ\doubleslashz\cG(P)$, which is
a symplectic quotient.\\
(2) $(\cM^\hit(\tP)^\circ)^\tau$ is K\"ahler and totally geodesic in
$\cM^\hit(\tP)^\circ$ with respect to $\bar J,\bar\upom_J$ and totally real
and Lagrangian with respect to $\bar I,\bar K$ and $\bar\upom_I,\bar\upom_K$.\\
(3) there is a local K\"ahler diffeomorphism from $\cM^\hit(P)^\circ$ to
$(\cM^\hit(\tP)^\circ)^\tau$.
\end{thm}

The theorem of Donaldson and Corlette in the non-orientable setup
(Theorem~\ref{thm:DC}) enable us to identify Hitchin's moduli space
associated to an orientable or non-orientable manifold with the moduli space
of flat connections and therefore the representation varieties.
Let $\Gam$ be a finitely generated group and let $G$ be a connected complex
semi-simple Lie group.
The representation variety, $\Hom(\Gam,G)\doubleslash G:=\Hom^\red(\Gam,G)/G$,
is the quotient of the space of reductive homomorphisms from $\Gam$ to $G$ by
the conjugation action of $G$.
When $\Gam$ is the fundamental group of a compact manifold $M$, the
representation variety is also called the Betti moduli space of $M$;
it is homeomorphic to the union of the de Rham moduli spaces $\cM^\dR(P)$
associated to principal $G$-bundles $P\to M$ of various topology.
When $M$ is non-orientable, let $\tGam$ be the fundamental group of the
oriented cover $\tM$.
Then there is a short exact sequence $1\to\tGam\to\Gam\to\bZ_2\to1$ and $\tau$
acts as an involution on the representation variety
$\Hom(\tGam,G)\doubleslash G$ (Lemma~\ref{lem:z2}).
We study the relation of representation varieties associated to $\Gam$ and
$\tGam$ from an algebraic point of view.
Let $PG=G/Z(G)$, where $Z(G)$ is the center of $G$.
Our main results are summarized in the following theorem.

\begin{thm}\label{thm:cover}
Let $G$ be a connected complex semi-simple Lie group.
Let $M$ be a compact non-orientable manifold and let $\tM$ be its oriented
cover on which there is a non-trivial deck transformation $\tau$.
Denote $\Gam=\pi_1(M)$ and $\tGam=\pi_1(\tM)$ with some chosen base points.
Then\\
(1) there exists a continuous map $L$ from $(\Hom^\good(\tGam,G)/G)^\tau$ to
$Z(G)/2Z(G)$.
Consequently,
$(\Hom^\good(\tGam,G)/G)^\tau=\bigcup_{r\in Z(G)/2Z(G)}\cN^\good_r$,
where $\cN^\good_r$ is the preimage of $r\in Z(G)/2Z(G)$.\\
(2) there exists a $|Z(G)/2Z(G)|$-sheeted Galois covering map from \\
$\Hom^\good_\tau(\Gam,G)/G$ to $\cN^\good_0$.\\
In particular, if $|Z(G)|$ is odd, then there exists a bijection from
$\Hom^\good_\tau(\Gam,G)/G$ to $(\Hom^\good(\tGam,G)/G)^\tau$.
The above statements are true if $\Hom^\good$ is replaced by $\Hom^\irr$.\\
If in addition $M=\Sig$ is a compact non-orientable surface and $G$ is
simple and simply connected, then\\
(3) there exists a surjective map from $(\Hom^\irr(\tGam,G)/G)^\tau$ to
$\Hom^\irr_\tau(\Gam,PG)/PG$ that maps $\cN^\irr_r$ to flat $PG$-bundles
on $\Sig$ whose topological type is given by $r\in Z(G)/2Z(G)\cong H^2(\Sig,Z(G))$.
In particular, $\cN^\irr_0$ maps to the topologically trivial flat
$PG$-bundles on $\Sig$.
\end{thm}

\vspace{-.25cm}

Here $\Hom^\good$, following the terminology of \cite{JM}, denotes the ``good''
part of the space of homomorphisms that are reductive and whose stabilizer is
$Z(G)$, whereas $\Hom^\irr$ is the space of homomorphisms whose composition with
the adjoint representation of $G$ is an irreducible representation
(see subsection~\ref{sec:red=red} for details).
$\Hom^\good_\tau(\Gam,G)$ is the set of homomorphisms from $\Gam$ to $G$ whose
restriction to $\tGam$ is ``good''.
$\Hom^\good_\tau(\Gam,G)\doubleslash G$ is not smooth in general, but contains
a smooth part $(\cM^\fla(P^\bC))^\circ$ (upon identification of moduli spaces).
By parts~(1) and (2) of the theorem, there is a local homeomorphism
$\Hom^\good_\tau(\Gam,G)/G\to(\Hom^\good(\tGam,G)/G)^\tau$ (see also
Corollary~\ref{cor:localdiff}), which in fact restricts to the local
diffeomorphism $\cM^\dR(P^\bC)^\circ\to(\cM^\dR(\tP^\bC)^\circ)^\tau$ in
part~(3) of Theorem~\ref{thm:Hitchin} but is now more accurately described
using representation varieties.
Also, for $\phi\in\Hom^\good(\tGam,G)$ such that
$[\phi]\in\Hom^\good(\tGam,G)/G$ is fixed by $\tau$, $L([\phi])$ is
the obstruction of extending $\phi$ to a representation of $\Gam$.
In the gauge-theoretic language, $\phi$ corresponds to a flat connection on
$\tM$ and represents a point fixed by $\tau$ in the de Rham moduli space
$\cM^\dR(\tP^\bC)$, while extension of $\phi$ to $\Gam$ means that the flat
connection on $\tM$ is the pull-back of a flat connection on $M$.
Flat connections on $\tM$ that are not pull-backs from $M$ correspond to flat
$PG$-bundles over $M$ (where $PG=G/Z(G)$).
This is shown in part~(3) of Theorem~\ref{thm:cover} and then discussed in
greater generality in the last section.

For example, let $G=SL(2,\bC)$, $M$ a compact nonorientable surface and $\tM$
its orientable double cover.
Then $(\Hom^\good(\pi_1(\tM),G)/G)^\tau$ is labeled by $Z(G)/2Z(G)=\bZ_2$,
i.e., $(\Hom^\good(\pi_1(\tM),G)/G)^\tau=\bigcup_{r\in \bZ_2}\cN^\good_r$.
An element of $(\Hom^\good(\pi_1(\tM),G)/G)^\tau$ is mapped by map $L$ in
Theorem~\ref{thm:cover}(1) (defined in Proposition~\ref{pro:pro1}) to the null
element of $\bZ_2$ if and only if it represents a flat connection on $\tM$
that is the pull-back of a flat connection on $M$.
The natural map from $\Hom^\good(\pi_1(M),G)/G$ to
$(\Hom^\good(\pi_1(\tM),G)/G)^\tau$ is not surjective;
it is a $\bZ_2$-sheeted Galois covering map onto $\cN^\good_0$,
and $\cN^\good_1$ is not in the image.
$\cN^\irr_0$ corresponds to the space of topologically trivial flat
$PSL(2,\bC)$-bundles over $M$ while $\cN^\irr_1$ corresponds to that of
topologically nontrivial flat $PSL(2,\bC)$-bundles over $M$.

The rest of this paper is organized as follows.
In Section~\ref{sec:gauge}, we review the basic setup in the orientable case
and explain the Donaldson-Corlette theorem for bundles over non-orientable
manifolds.
We then study finite dimensional symplectic and hyper-K\"ahler manifolds with
an involution and apply the results to the gauge theoretical setting to prove
Theorem~\ref{thm:Hitchin}.
In Section~\ref{sec:repr}, we study flat $G$-connections by representation
varieties.
We show that a flat connection on $M$ is reductive if and only if its pull-back
to $\tM$ is reductive.
We then define the continuous map in part~(1) of Theorem~\ref{thm:cover} and
prove the rest of the theorem.
In Section~\ref{sec:geom}, we relate the components $\cN^\good_r$ ($r\ne0$) in
Theorem~\ref{thm:cover} to $G$-bundles over $\tM$ admitting an involution up
to $Z(G)$.

We note that in order to study the moduli space of $G$-bundles over the nonorientable manifold $M$ itself, our involution is fixed-point free on $\tM$ and is the identity map on $G$.
During the revision of this paper, we came across a few related works.
We thank O.~Garc\'ia-Prada for pointing out to us the paper
\cite{BGH}, where their anti-holomorphic involution acts both on the manifold $\tM$ and on the
structure group $G$, thus resulting in a different fixed-point set
of the moduli space.
In a more recent paper \cite{BS1}, which overlaps with a special case of
part~(2) of our Theorem~\ref{thm:Hitchin} when $\tM$ is a surface, the
anti-holomorphic involution on the surface is allowed to have fixed points.

\section{The gauge-theoretic perspective}\label{sec:gauge}

\subsection{Basic setup in the orientable case}\label{sect:orient}

Let $K$ be a connected compact Lie group and let $G=K^\bC$ be its
complexification.
Given a principal $K$-bundle $P$ over a compact orientable manifold $M$,
$P^\bC=P\times_KG$ is a principal bundle whose structure group is $G$.
The set $\cA(P)$ of connections on $P$ is an affine space modeled on
$\Om^1(M,\ad P)$.
At each $A\in\cA(P)$, the tangent space is $T_A\cA(P)\cong\Om^1(M,\ad P)$.
The total space of the tangent bundle over $\cA(P)$ is
$T\cA(P)=\cA(P)\times\Om^1(M,\ad P)$.
At $(A,\psi)\in T\cA(P)$, the tangent space is
$T_{(A,\psi)}T\cA(P)\cong\Om^1(M,\ad P)^{\oplus2}$.
There is a translation invariant complex structure $J$ on $T\cA(P)$ given
by $J(\al,\vph)=(\vph,-\al)$.
The space $T\cA(P)$ can be naturally identified with $\cA(P^\bC)$, the set
of connections on $P^\bC\to M$, via $(A,\psi)\mapsto A-\ii\psi$, under which
$J$ corresponds to the complex structure on $\cA(P^\bC)$ induced by $G=K^\bC$.
The covariant derivative on $\Om^\bullet(M,\ad P^\bC)$ is $D:=d_A-\ii\psi$,
where $d_A$ denotes the covariant derivative of $A\in\cA(P)$ and $\psi$
acts by bracket.

The group of gauge transformations on $P$ is $\cG(P)\cong\Gam(M,\Ad P)$.
It acts on $\cA(P)$ via $A\mapsto g\cdot A$, where
$d_{g\cdot A}=g\circ d_A\circ g^{-1}$ and on $T\cA(P)$ via
$g\colon(A,\psi)\mapsto(g\cdot A,\Ad_g\psi)$.
Since the action of $\cG(P)$ on $T\cA(P)$ preserves $J$, there is a
holomorphic $\cG(P)^\bC$ action on $(T\cA(P),J)$.
In fact, the complexification $\cG(P)^\bC$ can be naturally identified
with $\cG(P^\bC)\cong\Gam(M,\Ad P^\bC)$, and the action of $\cG(P^\bC)$ on
$T\cA(P)$ corresponds to the complex gauge transformations on $\cA(P^\bC)$,
i.e., $g\in\cG(P^\bC)\colon D\mapsto g\circ D\circ g^{-1}$.
Let
\begin{align*}
\cA^\fla&(P^\bC)=\{A-\ii\psi\in\cA(P^\bC):F_{A-\ii\psi}=0\}  \\
&=\left\{(A,\psi)\in T\cA:F_A-\tfrac12[\psi,\psi]=0,d_A\psi=0\right\}
\end{align*}
be the set of flat connections on $P^\bC$.
Since the vanishing of $F_{A-\ii\psi}$ is a holomorphic condition,
$\cA^\fla(P^\bC)$ is a complex subset of $\cA(P^\bC)$; it is also
invariant under $\cG(P^\bC)$.
The holonomy group $\Hol(A)$ of $A\in\cA^\fla(P^\bC)$ can be identified as
a subgroup of $G$, up to a conjugation in $G$.
A flat connection $A$ on $P^\bC$ is {\em reductive} if the closure of $\Hol(A)$
in $G$ is contained in the Levi subgroup of any parabolic subgroup containing
$\Hol(A)$; let $\cA^{\fla,\red}(P^\bC)$ be the set of such.
It can be shown that a flat connection is reductive if and only if its
orbit under $\cG(P^\bC)$ is closed \cite{C}.
The {\em de Rham moduli space}, or the moduli space of reductive flat
connections on $P^\bC$, is
\[ \cM^\dR(P^\bC)=\cA^\fla(P^\bC)\doubleslash\cG(P^\bC)
   =\cA^{\fla,\red}(P^\bC)/\cG(P^\bC).  \]
It has an induced complex structure $\bar J$ on its smooth part.

Assume that $M$ has a Riemannian structure and choose an invariant inner
product $(\cdot,\cdot)$ on the Lie algebra $\fk$ of $K$.
Then there is a symplectic structure on $T\cA(P)$, with which $J$ is
compatible, given by
\begin{equation}\label{eqn:om-J}
\upom_J((\al_1,\vph_1),(\al_2,\vph_2))
=\int_M(\vph_2,\wedge*\al_1)-(\vph_1,\wedge*\al_2),
\end{equation}
where $\al_1,\al_2,\vph_1,\vph_2\in\Om^{1,0}(M,\ad P)$, such that
$(T\cA(P),\upom_J)$ is K\"ahler.
The subset $\cA^\fla(P^\bC)$ is K\"ahler in $\cA(P^\bC)\cong T\cA(P)$.
We identify the Lie algebra $\Lie(\cG(P))\cong\Om^0(M,\ad P)$ with its dual
by the inner product on $\Om^0(M,\ad P)$.
The action of $\cG(P)$ on $(T\cA(P),\upom_J)$ is Hamiltonian, with moment map
\begin{equation}\label{eqn:mu-J}
\upmu_J(A,\psi)=d_A^*\psi\in\Om^0(M,\ad P).
\end{equation}
Let
\begin{align*}
&\cA^\hit(P)=\cA^\fla(P^\bC)\cap\upmu_J^{-1}(0)   \\
=&\big\{(A,\psi)\in T\cA:F_A-\tfrac12[\psi,\psi]=0,d_A\psi=0,d_A^*\psi=0\big\},
\end{align*}
the set of pairs $(A,\psi)$ satisfying Hitchin's equations~\eqref{H-eq},
and let the quotient space $\cM^\hit(P)=\cA^\hit(P)/\cG(P)$ be {\em Hitchin's moduli space}.
A theorem of Donaldson \cite{Do} and Corlette \cite{C} states that if $M$
is compact and if the structure group $G$ is semisimple, then
$\cM^\hit(P)\cong\cM^\dR(P^\bC)$.

Suppose that $M$ is a compact K\"ahler manifold of complex dimension $n$
and let $\om$ be the K\"ahler form on $M$.
Then there is a complex structure on $T\cA(P)$ given by
\[ I\colon(\al,\vph)\mapsto\frac1{(n-1)!}*(\om^{n-1}\wedge(\al,-\vph))
=\frac1{(n-1)!}\,\Lam^{n-1}(*\al,-*\vph), \]
where $(\al,\vph)\in\Om^1(M,\ad P)^{\oplus2}\cong T_{(A,\psi)}T\cA(P)$
and the map
$$\Lam\colon\Om^\bullet(M,\ad P)\to\Om^{\bullet-2}(M,\ad P)$$ is the
contraction by $\om$.
With respect to $I$, we have
$$T_{(A,\psi)}^{1,0}T\cA(P)\cong\Om^{0,1}
(M,\ad P^\bC)\oplus\Om^{1,0}(M,\ad P^\bC)$$
for any $(A,\psi)\in T\cA(P)$.
This complex structure $I$ is compatible with a symplectic form $\upom_I$
on $T\cA(P)$ given by
\[ \upom_I((\al_1,\vph_1),(\al_2,\vph_2))=\int_M\frac{\om^{n-1}}{(n-1)!}
   \wedge\big((\al_1,\wedge\al_2)-(\vph_1,\wedge\vph_2)\big), \]
where $\al_1,\al_2,\vph_1,\vph_2\in\Om^1(M,\ad P)$.
The action of $\cG(P)$ on $T\cA(P)$ is also Hamiltonian with respect to
$\upom_I$ and the moment map is
\[ \upmu_I(A,\psi)=\Lam\,(F_A-\tfrac12[\psi,\psi])\in\Om^0(M,\ad P), \]
where $F_A\in\Om^2(M,\ad P)$ is the curvature of $A$.
Since the action of $\cG(P)$ on $T\cA(P)$ preserves $I$, there is a holomorphic
$\cG(P^\bC)$ action on $(T\cA(P),I)$.
For any $(A,\psi)\in T\cA(P)$, write $\psi=\ii(\phi-\phi^*)$, where
$\phi\in\Om^{1,0}(M,\ad P^\bC)$, $\phi^*\in\Om^{0,1}(M,\ad P^\bC)$.
Here $\phi\mapsto\phi^*$ is induced by the conjugation on $G=K^\bC$ preserving
the compact form $K$.
Then $D=d_A-\ii\psi=D'+D''$, where $D'=\partial_A-\phi^*$,
$D''=\bar\partial_A+\phi$.
The action of $\cG(P^\bC)$ on $T\cA(P)\cong\cA(P^\bC)$ can be described by
$g\in\cG(P^\bC)\colon D''\mapsto g\circ D''\circ g^{-1}$.

Let $\cA^\hig(P^\bC)$ be the set of Higgs pairs $(A,\phi)$, i.e., $A\in\cA(P)$
and $\phi\in\Om^{1,0}(M,\ad P^\bC)$ satisfying $(D'')^2=0$, or
\[ \bar\partial_A^2=0,\quad \bar\partial_A\phi=0,\quad [\phi,\phi]=0. \]
Then $\cA^\hig(P^\bC)$ is a K\"ahler subspace of $\cA(P^\bC)\cong T\cA(P)$
respect to $I$.
Let $\cA^\sst(P^\bC)$ be the set of semistable Higgs pairs
and let $\cA^\pst(P^\bC)$ be the set polystable Higgs pairs.
(The notions of stable, semistable and polystable Higgs pairs were introduced
in \cite{Hi,S92,S94}.)
The {\em moduli space of polystable Higgs pairs} or
the {\em Dolbeault moduli space} is
$$ \cM^\Dol(P^\bC)=(\cA^\hig(P^\bC)\cap\cA^\sst(P^\bC))\doubleslash\cG(P^\bC)
   =(\cA^\hig(P^\bC)\cap\cA^\pst(P^\bC))/\cG(P^\bC). $$
It has a complex structure induced by $I$.
It can be shown \cite[Lemma~1.1]{S92} that $\cA^\hit(P)
=\cA^\fla(P^\bC)\cap\upmu_J^{-1}(0)=\cA^\hig(P^\bC)\cap\upmu_I^{-1}(0)$.
A theorem of Hitchin \cite{Hi} and Simpson \cite{S88} states that if $M$ is
compact and K\"ahler and the bundle $P$ has vanishing first and second Chern
classes, then $\cM^\hit(P)\cong\cM^\Dol(P^\bC)$.

There is a third complex structure on $T\cA(P)$ defined by
\[ K=IJ=-JI\colon(\al,\vph)\mapsto\frac1{(n-1)!}*(\om^{n-1}\wedge(\vph,\al))
=\frac1{(n-1)!}\,\Lam^{n-1}(*\vph,*\al), \]
which is compatible with the symplectic form
\[ \upom_K((\al_1,\vph_1),(\al_2,\vph_2))=\int_M\frac{\om^{n-1}}{(n-1)!}
   \wedge\big((\al_1,\wedge\vph_2)-(\al_2,\wedge\vph_1)\big). \]
The action of $\cG(P)$ on $T\cA(P)$ is Hamiltonian with respect to $\upom_K$
and the moment map is
\[ \upmu_K(A,\psi)=\Lam\,(d_A\psi)\in\Om^0(M,\ad P).  \]
Moreover, the action preserves $K$ and therefore extends to another
holomorphic action of $\cG(P)^\bC$.
The three complex structures $I,J,K$ define a hyper-K\"ahler structure on
$T\cA(P)$.
Since the action of $\cG(P)$ on $T\cA(P)$ is Hamiltonian with respect to all
three symplectic forms, we have a hyper-K\"ahler moment map
$\upmu=(\upmu_I,\upmu_J,\upmu_K)\colon T\cA(P)\to(\Om^0(M,\ad P))^{\oplus3}$.
The {\em hyper-K\"ahler quotient} \cite{HKLR} is
$\cM^\HK(P)=\upmu^{-1}(0)/\cG(P)$, with complex structures
$\bar I,\bar J,\bar K$ and symplectic forms
$\bar\upom_I,\bar\upom_J,\bar\upom_K$.
By the theorems of Donaldson-Corlette and of Hitchin-Simpson, the Hitchin
moduli space $\cM^\hit(P)$ is a complex space with respect to both $\bar I$
and $\bar J$.
Therefore $\cM^\hit(P)$ is a hyper-K\"ahler subspace in $\cM^\HK(P)$
\cite[Theorem~8.3.1]{Fu}.

When $M=\Sig$ is an orientable surface,
$\Lam\colon\Om^2(\Sig,\ad P)\to\Om^0(\Sig,\ad P)$ is an isomorphism.
So $\cA^\hit(P)=\cA^\fla(P^\bC)\cap\upmu_J^{-1}(0)
=\cA^\hig(P^\bC)\cap\upmu_I^{-1}(0)$ coincides with
$\upmu^{-1}(0)=\upmu_I^{-1}(0)\cap\upmu_J^{-1}(0)\cap\upmu_K^{-1}(0)$.
Thus the moduli spaces
$\cM^\hit(P)\cong\cM^\dR(P^\bC)\cong\cM^\Dol(P^\bC)$ coincide with the
hyper-K\"ahler quotient $\cM^\HK(P)$ \cite{Hi}.

\subsection{Moduli space of Hitchin's equations on a non-orientable manifold}
\label{sec:DC}

Now suppose $M$ is a compact non-orientable manifold.
Let $\pi\colon\tM\to M$ be its oriented cover and let
$\tau\colon\tM\to\tM$ be the non-trivial deck transformation.
Given a principal $K$-bundle $P\to M$, let $\tP=\pi^*P\to\tM$ be
its pull-back to $\tM$.
Since $\pi\circ\tau=\pi$, the $\tau$ action can be lifted to $\tP=\tM\times_MP$
as a $K$-bundle involution (i.e., the lifted involution commutes with the right
$K$-action on $\tP$), and hence to the associated bundles $\Ad\tP$ and $\ad\tP$.
Consequently, $\tau$ acts on the space of connections $\cA(\tP)$ by pull-back
$A\mapsto\tau^*A$ and on the group of gauge transformations $\cG(\tP)$ by
$g\mapsto\tau^*g:=\tau^{-1}\circ g\circ\tau$.
The $\tau$-invariant subsets are $(\cA(\tP))^\tau\cong\cA(P)$ and
$(\cG(\tP))^\tau\cong\cG(P)$.
In fact, the inclusion map $\cA(P)\hookrightarrow\cA(\tP)$ onto the
$\tau$-invariant part is the pull-back via $\pi$ of connections on $P$ to
those on $\tP$.
Since $\cA(\tP)$ is an affine space modeled on $\Om^1(\tM,\ad\tP)$, the
differential $\tau_*$ of $\tau\colon\cA(\tP)\to\cA(\tP)$ can be identified
with a linear involution on $\Om^1(\tM,\ad\tP)$ given by $\al\mapsto\tau^*\al$.

A Riemannian metric on a non-orientable manifold $M$ pulls back to a Riemannian
metric on $\tM$.
Assuming that $M$ is compact, we define an inner product on the space
$\Om^\bullet(M)$ of differential forms on $M$ by
\[ \bra\al,\beta\ket=\tfrac12\int_{\tM}\pi^*\al\wedge\tilde*\,\pi^*\beta \]
for $\al,\beta\in\Om^\bullet(M)$, where $\tilde*$ is the Hodge star operator
on $\tM$.
Alternatively, the Hodge star $*$ on $M$ maps a form on $M$ to one valued in
the orientation line bundle over $M$, and if $\al,\beta$ are of the same degree,
then $\al\wedge*\beta$ is a top-degree form on $M$ valued in the orientation
line bundle, which can be integrated over $M$.
We still have $\bra\al,\beta\ket=\int_M\al\wedge*\beta$.
More generally, there is an inner product on the space $\Om^\bullet(M,\ad P)$
of forms valued in $\ad P$.
Therefore $\cA(P)$ admits a Riemannian structure, which is half of the
restriction of the Riemannian structure on $\cA(\tP)$ to the $\tau$-invariant
subspace $(\cA(\tP))^\tau\cong\cA(P)$.

Consider the tangent bundle $T\cA(\tP)=\cA(\tP)\times\Om^1(\tM,\ad\tP)$ of
$\cA(\tP)$.
It has a $\tau$-action given by
$\tau\colon(A,\psi)\mapsto(\tau^*A,\tau^*\psi)$, which is holomorphic with
respect to the complex structure $J$.
Therefore the fixed point set $(T\cA(\tP))^\tau\cong T\cA(P)$ is a complex
subspace in $T\cA(\tP)\cong\cA(\tP^\bC)$.
With respect to the induced Riemannian structure on $T\cA(\tP)$,
$\tau\colon T\cA(\tP)\to T\cA(\tP)$ is an isometry.
Since $\tau$ also acts holomorphically on $\cA(\tP^\bC)\cong T\cA(\tP)$,
$(T\cA(\tP))^\tau$ is a K\"ahler and totally geodesic subspace in
$T\cA(\tP)\cong\cA(\tP^\bC)$.
Moreover, $\cA^\fla(P^\bC)\cong(\cA^\fla(\tP^\bC))^\tau$ is also K\"ahler and
totally geodesic in $\cA^\fla(\tP^\bC)$.
We summarize the above discussion in the following lemma.

\begin{lem}
Given a compact non-orientable manifold $M$ with oriented double cover
$\pi\colon\tM\to M$ and a principal $K$-bundle $P\to M$, the non-trivial
deck transformation $\tau$ on $\tM$ lifts to an involution (also denoted
by $\tau$) on $\tP=\pi^*P$ and acts as involutions on the space of
connections $\cA(\tP)$ and on $T\cA(\tP)\cong\cA(\tP^\bC)$.
Moreover, the $\tau$-invariant subspaces $\cA(\tP^\bC)^\tau\cong\cA(P^\bC)$
and $\cA^\fla(\tP^\bC)^\tau\cong\cA^\fla(P^\bC)$ are K\"ahler and totally
geodesic subspaces in $\cA(\tP^\bC)\cong T\cA(\tP)$ and $\cA^\fla(\tP^\bC)$,
respectively.
\end{lem}

On a non-orientable manifold $M$, we still have Hitchin's equations
\eqref{H-eq}.
Here $d_A^*$ is defined as the (formal) adjoint of $d_A$ with respect to
the inner products on $\Om^\bullet(M,\ad P)$.
Alternatively, $d_A^*$ is the first order differential operator on $M$ such
that on any orientable open set in $M$, $d_A^*=*^{-1}d_A\,*$; the latter is
actually independent of the choice of local orientation.
Yet another but related way to explain the operator $d_A^*$ is to consider
the Hodge star operator $*$ on a non-orientable manifold $M$ as a map from
differential forms to those valued in the orientation bundle over $M$.
Since the latter is a flat real line bundle, $d_A^*=*^{-1}d_A*$ maps
$\Om^1(M,\ad P)$ to $\Om^0(M,\ad P)$.
Finally, $d_A^*$ can be defined as $(\pi^*)^{-1}\circ d_{\pi^*A}^*\circ\pi^*$.
Here $d_{\pi^*A}^*=*^{-1}d_{\pi^*A}\,*$ holds globally on $\tM$ and
$\pi^*\colon\Om^\bullet(M,\ad P)\to\Om^\bullet(\tM,\ad\tP)$ is injective.
Let
\[ \cA^\hit(P)
:=\{(A,\psi)\in T\cA:F_A-\tfrac12[\psi,\psi]=0,d_A\psi=0,d_A^*\psi=0\}. \]
It is clear that $\cA^\hit(P)=(\cA^\hit(\tP))^\tau$.

The notion of reductive connections on $P$ does not depend on the
orientability of $M$, and we still have the moduli space of flat
connections $\cM^\fla(P^\bC)=\cA^{\fla,\red}(P^\bC)/\cG(P^\bC)$.
Let $\cM^\hit(P)=\cA^\hit(P)/\cG(P)$ be \\
Hitchin's moduli space.
The following is the Donaldson-Corlette theorem that also applies to the case
when $M$ is non-orientable.
Equivalently, there exists a unique reduction of structure group from $G$
to $K$ admitting a solution to Hitchin's equations.

\begin{thm}\label{thm:DC}
Let $M$ be a compact non-orientable Riemannian manifold.
Then for every reductive flat connection $D$ on $P^\bC$, there exists a gauge
transformation $g\in\cG(P^\bC)$ (unique up to $\cG(P)$ and the stabilizer of
$D$) such that $g\cdot D=d_A-\ii\psi$ with $(A,\psi)\in\cA^\hit(P)$.
As a consequence, we have a homeomorphism $\cM^\dR(P^\bC)\cong\cM^\hit(P)$.
\end{thm}

We now explain that Corlette's proof in \cite{C} applies to the case when $M$
is non-orientable.
There is a symplectic form $\upom_J$ on $T\cA(P)$, still given
by\eqref{eqn:om-J}, which is half of the restriction of the symplectic form on
$T\cA(\tP)$.
The action of $\cG(P)$ on $T\cA(P)$ is Hamiltonian, and the moment map remains
\eqref{eqn:mu-J}.
Recall Corlette's flow equations on the space of flat connections.
Let $D=d_A-\ii\psi$ be a flat connection of the $G=K^\bC$ bundle $P^\bC\to M$.
Then the flow equations are
\begin{equation}\label{eqn:connection-flow}
\frac{\partial D}{\partial t}=-D\upmu_J(D).
\end{equation}
Equivalently, one can look for a flow of the form $g(t)\cdot D_0$ and solve
for $g(t)\in\cG(\tP^\bC)$ using (cf.~\cite[p.~369]{C})
\begin{equation}\label{eqn:gauge-flow}
\frac{\partial g}{\partial t}g^{-1}=-\ii\upmu_J(g\cdot D_0).
\end{equation}
Corlette shows in \cite{C} that we have existence and uniqueness of solutions
to \eqref{eqn:connection-flow} and \eqref{eqn:gauge-flow} for all time.
If the initial condition is a reductive flat connection, then there is a
sequence converging to a solution to $\upmu_J(D)=0$.
Also, the limit is gauge equivalent to the initial flat reductive connection
\cite{C}.
These arguments are valid when $M$ is non-orientable.

We remark that Theorem~\ref{thm:DC} for non-orientable manifolds also follows
from the result of the orientable double cover.
A flat connection on $P$ is reductive if and only if the pull-back $\pi^*A$
is a flat reductive connection on $\tP$.
(We defer the proof of this statement to Corollary~\ref{cor:red=red}.)
For the bundle $\tP\to\tM$, it is easy to check that the right-hand sides of
\eqref{eqn:connection-flow} and \eqref{eqn:gauge-flow} define $\tau$-invariant
vector fields on $\cA(\tP^\bC)$ and $\cG(\tP^\bC)$, respectively.
Since the space $(\cA^\fla(\tP^\bC))^\tau$ of $\tau$-invariant connections
is closed in $\cA^\fla(\tP^\bC)$ and the space $(\cG(\tP^\bC))^\tau$ of
$\tau$-invariant gauge transformations is closed in $\cG(\tP^\bC)$,
Corlette's results on the limit of the flow restrict to the $\tau$-invariant
subset as well.
That is, the flow on the space of connections is contained in the
$\tau$-invariant subset and the limit is a $\tau$-invariant solution to
Hitchin's equation.
Similarly, the gauge transformation relating to the initial condition is
contained in the $\tau$-invariant part of the group of gauge transformations,
and the limit is $\tau$-invariant.

\subsection{The Hitchin moduli space and the hyper-K\"ahler quotient}
\label{sec:Hitchin}

Now consider a compact non-orientable manifold $M$.
Suppose its oriented cover $\tM$ is a K\"ahler manifold of complex dimension
$n$.
Let $\om$ be the K\"ahler form on $\tM$.
Throughout this subsection, we assume that $n$ is odd and the deck
transformation $\tau$ on $\tM$ is an anti-holomorphic involution
such that $\tau^*\om=-\om$.
Then $\tau^*\om^n=-\om^n$, which is consistent with the requirement that $\tau$
is orientation reversing.
The $\tau$-action on $T\cA(\tP)=\cA(\tP)\times\Om^1(M,\ad P)$,
$\tau\colon(A,\psi)\mapsto(\tau^*A,\tau^*\psi)$, is an isometry and its
differential $\tau_*\colon\Om^1(M,\ad P)^{\oplus2}\to\Om^1(M,\ad P)^{\oplus2}$
is $\tau_*\colon(\al,\vph)\mapsto(\tau^*\al,\tau^*\vph)$.
It is easy to see that $\tau_*\circ I=-I\circ\tau_*$ since $\tau$ reverses
the orientation of $M$ and that $\tau_*\circ K=-K\circ\tau_*$ since $K=IJ$.
So $\tau$ acts as an anti-holomorphic involution with respect to both $I$ and
$K$, and $\tau^*\upom_I=-\upom_I$, $\tau^*\upom_K=-\upom_K$.
Moreover, since the moment maps $\upmu_I$ and $\upmu_K$ on $T\cA(\tP)$ involve
the contraction $\Lam$ by $\om$, they satisfy
$\tau^*(\upmu_I(A,\psi))=-\upmu_I(\tau^*A,\tau^*\psi)$,
$\tau^*(\upmu_K(A,\psi))=-\upmu_K(\tau^*A,\tau^*\psi)$ for all
$(A,\psi)\in T\cA(\tP)$.
The fixed point set $(\cA(\tP))^\tau$ is totally real with respect to the
complex structures $I$ and $K$, and Lagrangian with respect to the symplectic
forms $\upom_I$ and $\upom_K$ \cite{Me,Du,OS}.

A flat connection $D=d_A-\ii\psi$ on $\tP^\bC$ defines an elliptic complex
with $D_i\colon\Om^i(\tM,\ad\tP^\bC)\to\Om^{i+1}(\tM,\ad\tP^\bC)$.
Let $\cA^\fla(\tP^\bC)^\circ$ be the set of flat connections on $\tP^\bC$
such that (i) the stabilizer under the $\cG(\tP^\bC)$ action is $Z(G)$, and
(ii) the linearization $D_1$ of the curvature map surjects onto
$\ker D_2\cap\Om^2(\tM,[\ad\tP^\bC,\ad\tP^\bC])$.
Notice that when $M$ is a surface, condition~(i) implies (ii).
The method in \cite{Ki} and \cite[Chapter VII]{Ko} shows that
$\cA^\fla(\tP^\bC)^\circ$ is a smooth submanifold in $\cA(\tP^\bC)$,
and as the action of $\cG(\tP^\bC)/Z(G)$ on it is free, the subset
$\cM^\dR(\tP^\bC)^\circ:=(\cA^\fla(\tP^\bC)^\circ\cap\cA^{\fla,\red}(\tP^\bC))/
\cG(\tP^\bC)$ is in the smooth part of the moduli space $\cM^\dR(\tP^\bC)$
(see also \cite{Go} from the point of view of representation varieties).
The free action of $\cG(\tP^\bC)/Z(G)$ or $\cG(\tP)/Z(K)$ from condition~(i)
implies that $0$ is a regular value of $\upmu_J$ on $\cA^\fla(\tP^\bC)^\circ$,
and the subset
$\cM^\hit(\tP)^\circ:=\cA^\fla(\tP^\bC)^\circ\cap\upmu_J^{-1}(0)/\cG(\tP)$
is in the smooth part of Hitchin's moduli space $\cM^\hit(\tP)$ \cite{Hi}.
By the Donaldson-Corlette theorem, we have the homeomorphism
$\cM^\hit(\tP)^\circ\cong\cM^\dR(\tP^\bC)^\circ$.

On the other hand, for the non-orientable manifold $M$, let
$\cA^\fla(P^\bC)^\circ=\{A\in\cA(P^\bC):\pi^*A\in\cA^\fla(\tP^\bC)^\circ\}$,
$\cA^\hit(P)^\circ=\cA^\hit(P)\cap\cA^\fla(P^\bC)^\circ$. \\
Then $\cM^\hit(P)^\circ:=\cA^\hit(P)^\circ/\cG(P)$ is in the smooth part of
$\cM^\hit(P)$, but we will not consider here the smooth points of
$\cM^\hit(P)$ that are outside $\cM^\hit(P)^\circ$.
By Theorem~\ref{thm:DC} (the analog of the Donaldson-Corlette theorem for
non-orientable manifolds), we have a homeomorphism between
$\cM^\hit(P)^\circ$ and $\cM^\dR(P^\bC)^\circ:=
(\cA^\fla(P^\bC)^\circ\cap\cA^{\fla,\red}(P^\bC))/\cG(P^\bC)$.

We now study a general setting.
Let $(X,\om)$ be a finite dimensional symplectic manifold with a Hamiltonian
action of a compact Lie group $K$ and let $\mu\colon X\to\fk^*$ be the moment
map.
Suppose as in \cite{OS}, that there are involutions $\sig$ on $X$ and $\tau$
on $K$ such that $\sig(k\cdot x)=\tau(k)\cdot\sig(x)$ for all $k\in K$ and
$x\in X$.
Assume that $X^\sig$ is not empty.
Then $K^\tau$ acts on $X^\sig$.
We note that $\tau$ acts on $\fk$, $\fk^*$, and $K^\tau$ is a closed
Lie subgroup of $K$ with Lie algebra $\fk^\tau$.
Contrary to \cite{OS}, we assume that the action of $(K,K^\tau)$ on
$(X,X^\sig)$ is symplectic, i.e, we have $\sig^*\om=\om$ and
$\sig^*\mu=\tau\mu$.
Then $X^\sig$ is a symplectic submanifold in $X$.
Assume that $0$ is a regular value of $\mu$ and that $K$ acts on $\mu^{-1}(0)$
freely.
Since $\sig$ preserves $\mu^{-1}(0)$, it descends to a symplectic involution
$\bar\sig$ on the (smooth) symplectic quotient $X\doubleslashz K=\mu^{-1}(0)/K$
at level $0$, and $(X\doubleslashz K)^{\bar\sig}$ is a symplectic submanifold.

\begin{lem}\label{lem:fd}
In the above setting, the action of $K^\tau$ on $X^\sig$ is Hamiltonian and
the symplectic quotient is
$X^\sig\doubleslashz K^\tau=(\mu^{-1}(0)\cap X^\sig)/K^\tau$.
If $\mu^{-1}(0)\cap X^\sig\ne\emptyset$, then there exists a symplectic local
diffeomorphism from $X^\sig\doubleslashz K^\tau$ to
$(X\doubleslashz K)^{\bar\sig}$.
\end{lem}

\begin{proof}
Let $\fk=\fk^\tau\oplus\mathfrak{q}$ such that $\tau=\pm1$ on $\fk^\tau$,
$\mathfrak{q}$, respectively.
It is clear that the action of $K^\tau$ on $X^\sig$ is Hamiltonian and the
moment map $\mu_\tau$ is the composition
$X^\sig\hookrightarrow X\to\fk^*\to(\fk^\tau)^*$.
Since for any $x\in X^\sig$, $\bra\mu(x),\mathfrak{q}\ket=0$, we get
$\mu_\tau^{-1}(0)=\mu^{-1}(0)\cap X^\sig=(\mu^{-1}(0))^\sig$.
By the assumptions, $0$ is a regular value of $\mu_\tau$, the action of
$K^\tau$ on $\mu_\tau^{-1}(0)$ is free, and the symplectic quotient is
$X^\sig\doubleslashz K^\tau=(\mu^{-1}(0)\cap X^\sig)/K^\tau$.

For any $x\in X^\sig$, the map $\fk\to T_xX$ intertwines $\tau$ on $\fk$ and
$\sig$ on $T_xX$, and $T_x(K^\tau\cdot x)=(T_x(K\cdot x))^\sig$.
The inclusion $\mu_\tau^{-1}(0)\hookrightarrow\mu^{-1}(0)$ induces a natural
map $X^\sig\doubleslashz K^\tau\to(X\doubleslashz K)^{\bar\sig}$, whose
differentiation at $[x]$ is, after natural symplectic isomorphisms, the linear
map $(T_x\mu^{-1}(0))^\sig/(T_x(K\cdot x))^\sig
\to(T_x\mu^{-1}(0)/T_x(K\cdot x))^{\bar\sig}$.
The latter is clearly injective; to show surjectivity, we note that for any
$V\in T_x\mu^{-1}(0)$, if
$V+T_x(K\cdot x)\in(T_x\mu^{-1}(0)/T_x(K\cdot x))^{\bar\sig}$, then it is the
image of $\frac12(V+\sig V)+(T_x(K\cdot x))^\sig$.
The map $X^\sig\doubleslashz K^\tau\to(X\doubleslashz K)^{\bar\sig}$ is a local
diffeomorphism; it is symplectic because the above linear map is so for each
$x\in\mu_\tau^{-1}(0)$.
\end{proof}

Now let $X$ be a hyper-K\"ahler manifold with complex structures $J_i$ and
symplectic structures $\om_i$ ($i=1,2,3$).
Suppose $K$ acts on $X$ and the action is Hamiltonian with respect to all
$\om_i$.
Let $\mu=(\mu_1,\mu_2,\mu_3)\colon X\to(\fk^*)^{\oplus3}$ be the hyper-K\"ahler
moment map.
Assume that there are involutions $\sig$ on $X$ and $\tau$ on $K$ such that
$\sig(k\cdot x)=\tau(k)\cdot\sig(x)$ for all $k\in K$ and $x\in X$ and
$\sig^*J_i=(-1)^iJ_i$, $\sig^*\om_i=(-1)^i\om_i$, $\sig^*\mu_i=(-1)^i\tau\mu_i$
for $i=1,2,3$.
So the action of $(K,K^\tau)$ on $(X,X^\sig)$ is symplectic with respect to
$\om_2$ (as above) and anti-symplectic with respect to $\om_1,\om_3$ (as in
\cite{OS}).
Then $X^\sig$, if non-empty, is K\"ahler and totally geodesic in $X$ with
respect to $J_2,\om_2$ and is totally real and Lagrangian with respect to
$J_1,\om_1$ and $J_3,\om_3$.
If $0$ is a regular value of $\mu$ (i.e., $0$ is a regular value of each
$\mu_i$) and that $K$ acts on $\mu^{-1}(0)$ freely, then
$X\tripleslash K=\mu^{-1}(0)/K$ is the (smooth) hyper-K\"ahler quotient at
level $0$, which has complex structures $\bar J_i$ and symplectic structures
$\bar\om_i$ ($i=1,2,3$) \cite{HKLR}.

\begin{pro}\label{pro:fd}
In the above setting, let $Y=\mu_1^{-1}(0)\cap\mu_3^{-1}(0)$.
Then\\
1. $Y$ is a $\sig$-invariant K\"ahler submanifold in $X$ with respect to
$J_2,\om_2$ and the symplectic quotient
$Y^\sig\doubleslashz K^\tau=(\mu^{-1}(0))^\sig/K^\tau$ is K\"ahler;\\
2. $(X\tripleslash K)^{\bar\sig}$ is K\"ahler and totally geodesic in
$X\tripleslash K$ with respect to $\bar J_2,\bar\om_2$ and is totally real and
Lagrangian with respect to $\bar J_1,\bar J_3$ and $\bar\om_1,\bar\om_3$;\\
3. if $(\mu^{-1}(0))^\sig\ne\emptyset$, there is a K\"ahler (with respect to
$\bar J_2,\bar\om_2$) local diffeomorphism
$Y^\sig\doubleslashz K^\tau\to(X\tripleslash K)^{\bar\sig}$.
\end{pro}

\begin{proof}
1\&3. Let $\mu_c=\mu_3+\ii\mu_1\colon X\to\fk^{*\bC}$.
Then $\mu_c$ is holomorphic with respect to $J_2$ and is equivariant under
the action of $K$.
Since $0$ is a regular value of $\mu_c$, $Y=\mu_c^{-1}(0)$ is a smooth
K\"ahler submanifold in $X$ on which the action of $K$ is Hamiltonian.
Applying Lemma~\ref{lem:fd} to $Y$, we conclude that the action of $K^\tau$
on $Y^\sig$ is Hamiltonian and that $(\mu^{-1}(0))^\sig/K^\tau=
(\mu_2^{-1}(0)\cap Y^\sig)/K^\tau=Y^\sig\doubleslashz K^\tau$.
Moreover, there is a local diffeomorphism from $Y^\sig\doubleslashz K^\tau$
to $(Y\doubleslashz K)^{\bar\sig}=(X\tripleslash K)^{\bar\sig}$ which is
symplectic.
Since $K^\tau$ acts holomorphically on $(Y^\sig,J_2)$, the symplectic quotient
$Y^\sig\doubleslashz K^\tau$ is K\"ahler, and the above local diffeomorphism
is also K\"ahler.

2. Since $\sig$ preserves $\mu^{-1}(0)$, it descends to an involution
$\bar\sig$ on $X\tripleslash K$ such that $\bar\sig^*\bar J_i=(-1)^i\bar J_i$,
$\bar\sig^*\bar\om_i=(-1)^i\bar\om_i$ for $i=1,2,3$.
The result then follows.
\end{proof}

We now prove Theorem~\ref{thm:Hitchin}.

\begin{proof}
1\&3.
Note that $\cA^\fla(\tP^\bC)^\circ$ is a $\tau$-invariant K\"ahler submanifold
in $T\cA(\tP)\cong\cA(\tP^\bC)$.
Following \cite{AB,Hi}, we can apply the method in Lemma~\ref{lem:fd} to
$\cA^\fla(\tP^\bC)^\circ$ on which $\tau$ acts preserving $\upom_J$ and $J$.
Since $\tau$ also acts on $\cG(\tP)$ and $\cG(P)\cong(\cG(\tP))^\tau$,
$\cG(P)/Z(K)$ acts Hamiltonianly and freely on
$\cA^\fla(P^\bC)^\circ\cong(\cA^\fla(\tP^\bC)^\circ)^\tau$, which is K\"ahler
with respect to $J,\upom_J$.
Thus $\cM^\hit(P)^\circ=(\cA^\fla(P^\bC)^\circ\cap\upmu_J^{-1}(0))/\cG(P)
=\cA^\fla(P^\bC)^\circ\doubleslashz\cG(P)$ is a symplectic quotient.
Since the latter is non-empty, there is a local K\"ahler diffeomorphism
$\cM^\hit(P)^\circ\to(\cA^\fla(\tP^\bC)^\circ\doubleslashz\cG(\tP))^\tau
=(\cM^\hit(\tP)^\circ)^\tau$.

\noindent
2. The space $T\cA(\tP)\cong\cA(\tP^\bC)$ with $I,J,K$ is hyper-K\"ahler and
the action of $\cG(\tP)$ is Hamiltonian with respect to
$\upom_I,\upom_J,\upom_K$.
Let $(\upmu^{-1}(0))^\circ$ be the subset of $\upmu^{-1}(0)$ on which
$\cG(\tP)/Z(K)$ acts freely.
Then $\cM^\HK(\tP)^\circ:=(\upmu^{-1}(0))^\circ/\cG(\tP)$ is the smooth part
of the hyper-K\"ahler quotient $\cM^\HK(\tP)$.
The involutions $\tau$ on $\cA(\tP)$ and $\cG(\tP)$ satisfy the conditions of
Proposition~\ref{pro:fd}.
So $(\cM^\HK(\tP)^\circ)^\tau$ is K\"ahler and totally geodesic with respect
to $\bar J$ and $\bar\upom_J$, and totally real and Lagrangian with respect
to $\bar I,\bar K$ and $\bar\upom_I,\bar\upom_K$ in $\cM^\HK(\tP)^\circ$.
If $M$ is a non-orientable surface, then
$\upmu_I^{-1}(0)\cap\upmu_K^{-1}(0)=\cA^\fla(\tP^\bC)$ which implies that
 $\cM^\hit(\tP)^\circ=\cM^\HK(\tP)^\circ$.
In general, $\cM^\hit(\tP)^\circ$ is a $\tau$-invariant hyper-K\"ahler
submanifold in $\cM^\HK(\tP)^\circ$.
The results follow from $(\cM^\hit(\tP)^\circ)^\tau
=\cM^\hit(\tP)\cap(\cM^\HK(\tP)^\circ)^\tau$.
\end{proof}

\section{The representation variety perspective}\label{sec:repr}

\subsection{Representation variety and Betti moduli space}\label{sec:red=red}

Let $\Gam$ be a finitely generated group and let $G$ be a connected complex
Lie group.
Then $G$ acts on $\Hom(\Gam,G)$ by the conjugate action on $G$.
A representation $\phi\in\Hom(\Gam,G)$ is {\em reductive} if the closure of
$\phi(\Gam)$ in $G$ is contained in the Levi subgroup of any parabolic subgroup
containing $\phi(\Gam)$; let $\Hom^\red(\Gam,G)$ be the set of such.
The condition $\phi\in\Hom^\red(\Gam,G)$ is equivalent to the statement that
the $G$-orbit $G\cdot\phi$ is closed \cite{GM}.
It is also equivalent to the condition that the composition of $\phi$ with
the adjoint representation of $G$ is semi-simple
(see \cite[Section 3]{Richardson88} and \cite[Theorem 30]{Sikora12}).
The quotient
\[ \Hom(\Gam,G)\doubleslash G=\Hom^\red(\Gam,G)/G \]
is known as the {\em representation variety} or {\em character variety}.
A reductive representation $\phi\in\Hom^\red(\Gam,G)$ is {\em good} \cite{JM}
if its stabilizer $G_\phi=Z(G)$; let $\Hom^\good(\Gam,G)$ be the set of such.
On the other hand, $\phi\in\Hom(\Gam,G)$ is {\em Ad-irreducible} if its
composition with the adjoint representation of $G$ is an irreducible
representation of $\Gam$.
Let $\Hom^\irr(\Gam,G)$ be the set of such.
Notice that this set is empty unless $G$ is simple.
Clearly, $\Hom^\irr(\Gam,G)\subset\Hom^\good(\Gam,G)$.
In general, $\Hom^\good(\Gam,G)/G$ may not be smooth, but it is so when
$\Gam$ is the fundamental group of a compact orientable surface
\cite[Corollary~50]{Sikora12}.

Suppose $M$ is a compact manifold and $P^\bC\to M$ is a principal $G$-bundle
over $M$.
Choose a base point $x_0\in M$ and let $\Gam=\pi_1(M,x_0)$ be the fundamental group.
Then $\Hom(\Gam,G)\doubleslash G$ is known as the {\em Betti moduli space}
\cite{S92}, denoted by $\cM^\bet(P^\bC)$.
The identification $\cM^\dR(P^\bC)\cong\cM^\bet(P^\bC)$, which we recall
briefly now, is well known.
Given a flat connection, let $T_\al\colon P_{\al(0)}\to P_{\al(1)}$ be the
parallel transport along a path $\al$ in $M$.
Fix a point $p_0\in P_{x_0}$ in the fibre over $x_0$.
For $a\in\pi_1(M,x_0)$, choose a loop $\al$ based at $x_0$ representing $a$,
then $\phi(a)$ is the unique element in $G$ defined by
$T_\al(p_0)=p_0\phi(a)^{-1}$.
If we choose another point in the fibre over $x_0$, then $\phi$ differs by
a conjugation.
Finally, the flat connection is reductive if and only if the corresponding
element in $\Hom(\Gam,G)$ is reductive.
Upon identification of the de Rham moduli space $\cM^\dR(P^\bC)$ and the Betti
moduli spaces $\cM^\bet(P^\bC)=\Hom(\Gam,G)\doubleslash G$, the subset
$\Hom^\good(\Gam,G)/G$ contains the smooth part $\cM^\dR(P^\bC)^\circ$
introduced in subsection~\ref{sec:Hitchin};
they are equal when $M$ is a compact orientable surface.

If $M$ is non-orientable and $\pi\colon\tM\to M$ is the oriented cover,
we choose a base point $\tilde x_0\in\pi^{-1}(x_0)$ and let
$\tGam=\pi_1(\tM,\tilde x_0)$.
Then there is a short exact sequence
\[ 1\to\tGam\to\Gam\to\bZ_2\to1 \]
and $\tGam$ can be identified with an index $2$ subgroup in $\Gam$.
In the rest of this section, we will study the relation of the representation
varieties $\Hom(\Gam,G)\doubleslash G$ and $\Hom(\tGam,G)\doubleslash G$ or
the Betti moduli spaces $\cM^\bet(P^\bC)$ and $\cM^\bet(\tP^\bC)$.
Some of the results, when $M$ is a compact non-orientable surface, appeared in
\cite{Ho}, which used different methods.

We first establish a useful fact that was used in subsection~\ref{sec:DC}.

\begin{lem}\label{lem:reductive}
Suppose $\Gam$ is a finitely generated group and $\tGam$ is an index~$2$
subgroup in $\Gam$.
Let $G$ be a connected, complex reductive Lie group.
Then $\phi\in\Hom(\Gam,G)$ is reductive if and only if the restriction
$\phi|_{\tGam}\in\Hom(\tGam,G)$ is reductive.
\end{lem}

\begin{proof}
Recall that $\phi\in\Hom(\Gam,G)$ is reductive if and only if the composition
$\Ad\circ\phi$ is a semisimple representation on $\mathfrak{g}$.
Similarly, $\phi|_{\tGam}$ is reductive if and only if
$\Ad\circ\phi|_{\tGam}$ is semisimple.
By $\Gam/\tGam\cong\bZ_2$ and \cite{Cl}, \cite[Chap.~3, \S9.8, Lemme~2]{Bou},
$\Ad\circ\phi$ is semisimple if only if $\Ad\circ\phi|_{\tGam}$ is so.
The result then follows.
\end{proof}

\begin{cor}\label{cor:red=red}
Let $G$ be a connected, complex reductive Lie group.
Suppose $P$ is a principal $G$-bundle over a compact non-orientable manifold
$M$ whose oriented cover is $\pi\colon\tM\to M$.
Then a flat connection $A$ on $P$ is reductive if and only if the pull-back
$\pi^*A$ is a flat reductive connection on $\tP:=\pi^*P$.
\end{cor}

\subsection{Representation varieties associated to an index $2$ subgroup}
\label{sec:pro1&2}

Let $\Gam$ be a finitely generated group and let $\tGam$ be an index~$2$
subgroup in $\Gam$.
Let $G$ be a connected complex Lie group and let $Z(G)$ be its center.
For any $c\in\Gam\setminus\tGam$, we have $\Ad_c|_{\tGam}\in\Aut(\tGam)$,
and the class $[\Ad_c|_{\tGam}]\in\Aut(\tGam)/\Inn(\tGam)$ is independent
of the choice of $c$.
So we have a homomorphism $\bZ_2\cong\{1,\tau\}\to\Aut(\tGam)/\Inn(\tGam)$
given by $\tau\mapsto[\Ad_c|_{\tGam}]$.

\begin{lem}\label{lem:z2}
$\bZ_2\cong\{1,\tau\}$ acts on $\Hom(\tGam,G)\doubleslash G$ and on
$\Hom^\good(\tGam,G)/G$.
\end{lem}

\begin{proof}
We define $\tau[\phi]=[\phi\circ\Ad_c]$ for any $\phi\in\Hom(\tGam,G)$.
The action is well-defined since if $[\phi']=[\phi]$, i.e.,
$\phi'=\Ad_g\circ\phi$ for some $g\in G$, then
$\phi'\circ\Ad_c=\Ad_g\circ\phi\circ\Ad_c\sim\phi\circ\Ad_c$.
The $\tau$-action is independent of the choice of $c$ because if
$c'\in\Gam\setminus\tGam$ is another element, then $c'c^{-1}\in\tGam$
and $\phi\circ\Ad_{c'}
=\Ad_{\phi(c'c^{-1})}\circ(\phi\circ\Ad_c)\sim\phi\circ\Ad_c$.
We do have a $\bZ_2$-action because
$\tau^2[\phi]=[\phi\circ\Ad_{c^2}]=[\Ad_{\phi(c^2)}\circ\phi]=[\phi]$.
Finally, if $\phi$ is in $\Hom^\red(\tGam,G)$ or $\Hom^\good(\tGam,G)$,
then so is $\phi\circ\Ad_c$.
Thus $\tau$ acts on $\Hom(\tGam,G)\doubleslash G$ and $\Hom^\good(\tGam,G)/G$.
\end{proof}

\begin{pro}\label{pro:pro1}
There exists a continuous map
\begin{equation}\label{eq:def_of_L}
L\colon(\Hom^\good(\tGam,G)/G)^\tau\to Z(G)/2Z(G).
\end{equation}
So $(\Hom^\good(\tGam,G)/G)^\tau=\bigcup_{r\in Z(G)/2Z(G)}\cN^\good_r$,
where $\cN^\good_r:=L^{-1}(r)$.
\end{pro}

\begin{proof}
If $\tau[\phi]=[\phi]$, then there exists $g\in G$ such that
$\phi\circ\Ad_c=\Ad_g\circ\phi$.
Since $c^2\in\tGam$, we have
$\Ad_{g^2}\circ\phi=\phi\circ\Ad_{c^2}=\Ad_{\phi(c^2)}\circ\phi$.
Thus $z:=g^2\phi(c^2)^{-1}\in G_\phi=Z(G)$.
If $[\phi']=[\phi]$, i.e., $\phi'=\Ad_h\circ\phi$ for some $h\in G$,
then $\phi'\circ\Ad_c=\Ad_{g'}\circ\phi'$ for $g'=\Ad_hg$.
Since $g'^2=\Ad_hg^2=z\Ad_h\phi(c^2)=z\phi'(c^2)$, we obtain
$(g')^2\phi'(c^2)^{-1}=z$.

If $\phi\circ\Ad_{c'}=\Ad_{g'}\circ\phi$ holds for different choices of
$c'\in\Gam\setminus\tGam$ and $g'\in G$, then
$z'=(g')^2\phi(c'^2)^{-1}\in Z(G)$ from the above discussion.
On the other hand, we have
$\Ad_{g^{-1}g'}\circ\phi=\Ad_{\phi(c^{-1}c')}\circ\phi$ as $c^{-1}c'\in\tGam$.
This gives us $t:=(g')^{-1}g\phi(c^{-1}c')\in G_\phi=Z(G)$.
We get
\begin{eqnarray*}\lefteqn{t^2(g')^2=(tg')^2=g\phi(c^{-1}c')g\phi(c^{-1}c')
=\Ad_g\phi(c^{-1}c')g^2\phi(c^{-1}c')}\\&& \quad
=\phi(\Ad_c(c^{-1}c'))z\phi(c^2)\phi(c^{-1}c')=\phi((c')^2)z,\end{eqnarray*}
i.e., $z'z^{-1}=t^{-2}\in2Z(G)$.
So the map $L\colon[\phi]\mapsto[z]\in Z(G)/2Z(G)$ is well-defined.

Since $\phi\in\Hom^\good(\tGam,G)$, the element $[g]\in G/Z(G)$ is uniquely
determined by and depends continuously on $\phi$.
Therefore $[z]\in Z(G)/2Z(G)$ depends continuously on
$[\phi]\in(\Hom^\good(\tGam,G)/G)^\tau$.
\end{proof}

If $\phi\in\Hom(\Gam,G)$ satisfies $\phi|_{\tGam}\in\Hom^\good(\tGam,G)$,
then $\phi\in\Hom^\good(\Gam,G)$.
However, $\phi\in\Hom^\good(\Gam,G)$ does not imply
$\phi|_{\tGam}\in\Hom^\good(\tGam,G)$.
Let
\[  \Hom^\good_\tau(\Gam,G)
=\{\phi\in\Hom(\Gam,G):\phi|_{\tGam}\in\Hom^\good(\tGam,G)\}. \]
We show that if $[\phi]\in(\Hom^\good(\tGam,G)/G)^\tau$, then $L([\phi])$
is the obstruction of extending $\phi$ to a representation of $\Gam$.

\begin{lem}\label{lem:pro1}
The restriction $R\colon[\phi]\mapsto[\phi|_{\tGam}]$ maps
$\Hom^\good_\tau(\Gam,G)/G$ surjectively to $\cN^\good_0$.
\end{lem}

\begin{proof}
First, the image $\im(R)\subset\cN^\good_0$ because for any
$\phi\in\Hom^\good_\tau(\Gam,G)$, $\phi|_{\tGam}\in\Hom^\good(\tGam,G)$ by definition,
so $(\phi|_{\tGam})\circ\Ad_c=\Ad_{\phi(c)}\circ\phi|_{\tGam}\sim\phi|_{\tGam}$
and $L([\phi|_{\tGam}])=[\phi(c)^2\phi(c^2)^{-1}]=0$.
We will show that in fact $\im(R)=\cN^\good_0$.
Let $\phi_0\in\Hom^\good(\tGam,G)$ such that $\tau[\phi_0]=[\phi_0]$ and
$L([\phi_0])=0$.
Then there exist $g\in G$ and $t\in Z(G)$ such that
$\phi_0\circ\Ad_c=\Ad_g\circ\phi_0$ and $g^2\phi(c^2)^{-1}=t^2$.
We can extend $\phi_0$ to $\phi\in\Hom(\Gam,G)$ which is uniquely determined
by the requirements $\phi|_{\tGam}=\phi_0$ and $\phi(c)=gt^{-1}$.
Since $\phi_0\in\Hom^\good(\tGam,G)$, $\phi\in\Hom_\tau^\good(\Gam,G)$ and
therefore $[\phi_0]\in\im(R)$.
\end{proof}

\begin{pro}\label{pro:pro2}
$R\colon\Hom^\good_\tau(\Gam,G)/G\to\cN^\good_0$ is a Galois covering map
whose structure group is $\{s\in Z(G):s^2=e\}$.
\end{pro}

\begin{proof}
We define an action of $\{s\in Z(G):s^2=e\}$ on $\Hom^\good_\tau(\Gam,G)$.
For any such $s$ and $\phi\in\Hom^\good_\tau(\Gam,G)$, we define $s\cdot\phi$
by $(s\cdot\phi)|_{\tGam}=\phi|_{\tGam}$ and
$(s\cdot\phi)|_{\Gam\setminus\tGam}=s(\phi|_{\Gam\setminus\tGam})$
the group multiplication.
It is clear that $s\cdot\phi\in\Hom(\Gam,G)$.
Moreover, since
$(s\cdot\phi)|_{\tGam}=\phi|_{\tGam}\in\Hom^\good(\tGam,G)$,
$s\cdot\phi\in\Hom^\good_\tau(\Gam,G)$. \\
Clearly, the action descends to a well-defined action on
$\Hom^\good_\tau(\Gam,G)/G$ by $s\cdot[\phi]=[s\cdot\phi]$
preserving the fibres of $R$.

We show that this action is free.
Suppose $s\cdot[\phi]=[\phi]$, then $s\cdot\phi=\Ad_h\circ\phi$ for some
$h\in G$.
Since $\phi|_{\tGam}=(s\cdot\phi)|_{\tGam}=
\Ad_h\circ(\phi|_{\tGam})\in\Hom^\good(\tGam,G)$,
we get $h\in Z(G)$ and hence $s\cdot\phi=\phi$.
Then $s\phi(c)=\phi(c)$ implies $s=e$.

It remains to show that the action is transitive on each fibre of $R$.
Let $[\phi],[\phi']\in\Hom^\good_\tau(\Gam,G)$ such that $R([\phi])=R([\phi'])$.
Then there exists an $h\in G$ such that
$\phi'|_{\tGam}=\Ad_h\circ(\phi|_{\tGam})$.
Thus
\begin{eqnarray*}
\Ad_{\phi'(c)h}\circ(\phi|_{\tGam})
&=&\Ad_{\phi'(c)}\circ(\phi'|_{\tGam})=(\phi'|_{\tGam})\circ\Ad_c\\
&=&\Ad_h\circ(\phi|_{\tGam})\circ\Ad_c=\Ad_{h\phi(c)}\circ(\phi|_{\tGam}).
\end{eqnarray*}
Hence $s:=\phi(c)^{-1}h^{-1}\phi'(c)h\in Z(G)$ since
$\phi|_{\tGam}\in\Hom^\good(\tGam,G)$.
Furthermore
\[ s^2=\phi(c)^{-1}sh^{-1}\phi'(c)h=\phi(c^{-2})h^{-1}\phi'(c^2)h
=\phi(c^{-2})\phi(c^2)=e. \]
Since we have $(s\cdot\phi)|_{\tGam}=\phi|_{\tGam}=\Ad_{h^{-1}}\circ(\phi'|_{\tGam})$
and $(s\cdot\phi)(c)=s\phi(c)=\phi(c)s=(\Ad_{h^{-1}}\circ\phi')(c)$, we get
$s\cdot\phi=\Ad_{h^{-1}}\circ\phi'$, or $[\phi']=[s\cdot\phi]$.
\end{proof}

\begin{cor}\label{cor:localdiff}
Under the above assumptions, there is a local homeomorphism from
$\Hom^\good_\tau(\Gam,G)/G$ to\\
 $(\Hom^\good(\tGam,G)/G)^\tau$, which restricts
to a local diffeomorphism on the smooth part.
If $|Z(G)|$ is odd, this local homeomorphism (diffeomorphism, respectively)
is a homeomorphism (diffeomorphism, respectively).
\end{cor}

\begin{proof}
The first statement follows easily from Propositions~\ref{pro:pro1} and
\ref{pro:pro2}.
If $|Z(G)|$ is odd, we get $Z(G)/2Z(G)\cong\{0\}$ and
$(\Hom^\good(\tGam,G)/G)^\tau=\cN^\good_0$ by Proposition~\ref{pro:pro1}.
Furthermore, since $\{s\in Z(G):s^2=e\}=\{e\}$, the covering map in
Proposition~\ref{pro:pro2} is a bijection.
\end{proof}

The involution $\tau$ also acts on $\Hom^\irr(\tGam,G)/G$.
Let
\[ \Hom^\irr_\tau(\Gam,G)
=\{\phi\in\Hom(\Gam,G):\phi|_{\tGam}\in\Hom^\irr(\tGam,G)\}. \]
By the same idea used in the proof of Propositions~\ref{pro:pro1} and
\ref{pro:pro2}, we get

\begin{cor}
If $G$ is simple, there exists a decomposition
\[ (\Hom^\irr(\tGam,G)/G)^\tau
   =\bigcup_{r\in Z(G)/2Z(G)} \cN^\irr_r,  \]
where $\cN^\irr_r=\cN^\good_r\cap(\Hom^\irr(\tGam,G)/G)^\tau$.
Furthermore, there exists a Galois covering map
$R\colon\Hom^\irr_\tau(\Gam,G)/G\to\cN^\irr_0$ with structure group
$\{s\in Z(G):s^2=e\}$.
If $|Z(G)|$ is odd, then there is a bijection from $\Hom^\irr_\tau(\Gam,G)/G$
to $(\Hom^\irr(\tGam,G)/G)^\tau$.
\end{cor}

The results in this subsection show parts~(1) and (2) of
Theorem~\ref{thm:cover}.

\subsection{The Betti moduli space associated to a non-orientable surface}
\label{sec:pro3}

By subsection~\ref{sec:pro1&2} or parts (1) and (2) of
Theorem~\ref{thm:cover}, we know that a representation
$\phi\in\Hom^\good(\tGam,G)$ such that $\tau[\phi]=[\phi]$ can
be extended to one on $\Gam$ if an only if $L([\phi])=0$.
When applied to $\Gam=\pi_1(M)$ and $\tGam=\pi_1(\tM)$, where $M$ is
non-orientable and $\tM$ is its oriented cover, we conclude that a
$\tau$-invariant flat bundle over the $\tM$ corresponding to
$\phi\in\Hom^\good(\tGam,G)$ is the pull-back of a flat bundle over
$M$ if and only if $L([\phi])=0$.
We now consider the example when $M=\Sig$ is a compact non-orientable surface,
in which case we can characterize all the components $\cN^\good_r$ explicitly.
The principal $G$-bundles on $\Sig$ are topologically classified by
$H^2(\Sig,\pi_1(G))\cong\pi_1(G)/2\pi_1(G)$ whereas those on the oriented
cover $\tSig$ are classified by $H^2(\tSig,\pi_1(G))\cong\pi_1(G)$.
The classes in these groups are the obstructions of lifting the structure
group $G$ of the bundles to its universal cover group.

A compact non-orientable surface $\Sig$ is of the form $\Sig^\ell_k$
($\ell\ge0$, $k=1,2$), the connected sum of $2\ell+k$ copies of $\bR P^2$.
Then $\tSig$ is a compact surface of genus $2\ell+k-1$.
For $k=1$, we have
\[ \pi_1(\Sig)=\big\bra\,a_i,b_i\,(1\le i\le \ell),c:
c^{-2}\textstyle\prod_{i=1}^\ell[a_i,b_i]\,\big\ket,\]
\[ \pi_1(\tSig)=\big\bra\,a_i,b_i,a'_i,b'_i\,(1\le i\le \ell):
\textstyle\prod_{i=1}^\ell[a_i,b_i]\prod_{i=1}^\ell[a'_i,b'_i]\,\big\ket. \]
The inclusion $\pi_1(\tSig)\to\pi_1(\Sig)$ is given by $a_i\mapsto a_i$,
$b_i\mapsto b_i$, $a'_i\mapsto\Ad_cb_i$, $b'_i\mapsto\Ad_ca_i$
($1\le i\le \ell$).
For $k=2$, we have
\[ \pi_1(\Sig)=\big\bra\,a_i,b_i\, (1\le i\le \ell),c,d:
d^{-1}cd^{-1}c^{-1}\textstyle\prod_{i=1}^\ell[a_i,b_i]\,\big\ket, \]
\[ \pi_1(\tSig)=\big\bra a_0,b_0,a_i,b_i,a'_i,b'_i\,(1\le i\le \ell):
\textstyle[a_0,b_0]\prod_{i=1}^\ell[a_i,b_i]
\prod_{i=1}^\ell[a'_i,b'_i]\,\big\ket. \]
The inclusion $\pi_1(\tSig)\to\pi_1(\Sig)$ is given by $a_0\mapsto d^{-1}$,
$b_0\mapsto c^2$, $a_i\mapsto a_i$, $b_i\mapsto b_i$,
$a'_i\mapsto\Ad_{d^{-1}c}b_i$, $b'_i\mapsto\Ad_{d^{-1}c}a_i$ ($1\le i\le \ell$).
In both cases, $c\in\pi_1(\Sig)\setminus\pi_1(\tSig)$.

While a flat $G$-bundle over $\Sig$ may be non-trivial, its pull-back to
$\tSig$ is always trivial topologically \cite{HL2}.
We assume that $G$ is semi-simple, simply connected and denote $PG=G/Z(G)$.
Then $\pi_1(PG)=Z(G)$ and we have $H^2(\Sig,\pi_1(PG))\cong Z(G)/2Z(G)$.
The map $$O\colon\Hom(\pi_1(\Sig),PG)/PG\to Z(G)/2Z(G)$$ that gives the
obstruction class can be explicitly described as follows \cite{HL1}.
Let $\phi\in\Hom(\pi_1(\Sig),PG)$.
For $k=1$, let
$\widetilde{\phi(a_i)}$, $\widetilde{\phi(b_i)}$, $\widetilde{\phi(c)}\in G$
be any lifts of $\phi(a_i)$, $\phi(b_i)$, $\phi(c)\in PG$, respectively.
Then $O([\phi])$ is the element in $Z(G)/2Z(G)$ represented by
$\widetilde{\phi(c)}{}^2(\prod_{i=1}^\ell
[\widetilde{\phi(a_i)},\widetilde{\phi(b_i)}])^{-1}\in Z(G)$.
(It is easy to check that the class in $Z(G)/2Z(G)$ is independent of the
lifts.)
The description of the case $k=2$ is similar.
Consequently, there is a decomposition
$$\Hom(\pi_1(\Sig),PG)/PG=\bigcup_{r\in Z(G)/2Z(G)}\cM_r,$$
where $\cM_r=O^{-1}(r)$.

Let $G\to PG$, $g\mapsto\bar g$ be the quotient map.
Denote the induced map by
$\Hom(\pi_1(\Sig),G)\to\Hom(\pi_1(\Sig),PG),\  \phi\mapsto\bar\phi.$
In this section, we need to be restricted to Ad-irreducible representations.
The reason is that $\phi$ is Ad-irreducible if and only if $\bar\phi$ is so,
whereas if $\phi$ is good, $\bar\phi$ is not necessarily so and its stabilizer
may be larger than $Z(G)$.
We have
$$\Hom^\irr_\tau(\pi_1(\Sig),PG)/PG=\bigcup_{r\in Z(G)/2Z(G)}\cM_r^\irr,$$
where $\cM_r^\irr=\cM_r\cap(\Hom^\irr_\tau(\pi_1(\Sig),PG)/PG)$.

\begin{lem}\label{lm:pro3-1}
There is a natural map
\[ \Ps\colon(\Hom^\irr(\pi_1(\tSig),G)/G)^\tau\to
\Hom^\irr_\tau(\pi_1(\Sig),PG)/PG \]
satisfying $L=O\circ\Ps$.
Consequently, $\Ps$ maps $\cN^\irr_r$ to $\cM_r^\irr$ for each
$r\in Z(G)/2Z(G)$.
\end{lem}

\begin{proof}
Given $[\phi]\in(\Hom^\irr(\pi_1(\tSig),G)/G)^\tau$, there exists $g\in G$
(which is unique up to $Z(G)$ since $G_\phi=Z(G)$) such that
$\Ad_g\circ\phi=\phi\circ\Ad_c$.
We define $\check\phi\in\Hom(\pi_1(\Sig),PG)$ by
$\check\phi|_{\pi_1(\tSig)}=\bar\phi$ and $\check\phi(c)=\bar g$.
The representation $\check\phi$ is a homomorphism because
$\check\phi(c)^2=\bar g^2=\bar\phi(c^2)$, which follows from the result
$z=g^2\phi(c^2)^{-1}\in Z(G)$ in Proposition~\ref{pro:pro1}.
Since $\bar\phi\in\Hom^\irr(\pi_1(\tSig),PG)$, we have
$\check\phi\in\Hom^\irr_\tau(\pi_1(\Sig),PG)$.
We define $\Ps$ by $\Ps([\phi])=[\check\phi]$.
To show that $O([\check\phi])=L([\phi])=[z]$, we work in the case $k=1$.
By using the respective lifts $\phi(a_i)$, $\phi(b_i)$, $g\in G$ of
$\check\phi(a_i)$, $\check\phi(b_i)$, $\check\phi(c)\in PG$, we get
$$O([\check\phi])
=[g^2(\prod_{i=1}^\ell[\phi(a_i),\phi(b_i)])^{-1}]=[g^2\phi(c^2)^{-1}]=[z],$$
where we have used the relation $\prod_{i=1}^\ell[\phi(a_i),\phi(b_i)]=c^2$
in $\pi_1(\tSig)$.
The case $k=2$ is similar.
\end{proof}

\begin{pro}\label{pro:pro3}
The map
$$\Ps\colon(\Hom^\irr(\pi_1(\tSig),G)/G)^\tau\to
\Hom^\irr_\tau(\pi_1(\Sig),PG)/PG$$ is surjective.
Consequently, $\Ps\colon\cN^\irr_r\to\cM^\irr_r$ is surjective for each
$r\in Z(G)/2Z(G)$.
\end{pro}

\begin{proof}
Let $[\phi]\in\Hom^\irr_\tau(\Sig,PG)/PG$.
Although $\phi(c)\in PG$, $\Ad_{\phi(c)}$ acts on $G$.
We show the case $k=1$ only.
Fix the lifts $\widetilde{\phi(a_i)}$, $\widetilde{\phi(b_i)}\in G$ of
$\phi(a_i)$, $\phi(b_i)\in PG$.
Define $\tphi\in\Hom(\pi_1(\tSig),G)$ by setting $\tphi(a_i)=\widetilde{\phi(a_i)}$,
$\tphi(b_i)=\widetilde{\phi(b_i)}$, $\tphi(a'_i)=\Ad_{\phi(c)}\tphi(b_i)$,
$\tphi(b'_i)=\Ad_{\phi(c)}\tphi(a_i)$, for $i=1,\dots,\ell$.
This indeed defines a representation because
\[ \prod_{i=1}^\ell[\tphi(a_i),\tphi(b_i)]
\prod_{i=1}^\ell[\tphi(a'_i),\tphi(b'_i)]
=\prod_{i=1}^\ell[\tphi(a_i),\tphi(b_i)]\,\Ad_{\phi(c)}
\prod_{i=1}^\ell[\tphi(b_i),\tphi(a_i)]\!\!=e. \]
The last equality is because $\prod_{i=1}^\ell[\tphi(a_i),\tphi(b_i)]\in G$
projects to $\phi(c)^2\in PG$.
Since $\phi$ is Ad-irreducible, so is $\tphi$.
$[\tphi]$ is $\tau$-invariant because
$\tphi\circ\Ad_c=\Ad_{\phi(c)}\circ\tphi$,
which can be checked on the generators:
$\tphi(\Ad_ca_i)=\tphi(b'_i)=\Ad_{\phi(c)}\tphi(a_i)$,
$\tphi(\Ad_ca'_i)=\Ad_{\phi(c^2)}\tphi(b_i)=\Ad_{\phi(c)}\tphi(a'_i)$, etc.
It is then obvious that $\Ps([\tphi])=[\phi]$.
\end{proof}

For the group $PG$, since $Z(PG)$ is trivial,
$(\Hom^\irr(\pi_1(\tSig),PG)/PG)^\tau$ does not decompose according to
Proposition~\ref{pro:pro1} and the map
$$\bar R\colon\Hom^\irr_\tau
(\pi_1(\Sig),PG)/PG\to(\Hom^\irr(\pi_1(\tSig),PG)/PG)^\tau$$ in
Proposition~\ref{pro:pro2} is bijective.
The map $\Ps$ is in fact the composition of
$(\Hom^\irr(\pi_1(\tSig),G)/G)^\tau\to(\Hom^\irr(\pi_1(\tSig),PG)/PG)^\tau$
(induced by $G\to PG$) followed by $\bar R^{-1}$.
So for each $r\in Z(G)/2Z(G)$, the component $\cN^\irr_r$ of the fixed point
set $(\Hom^\irr(\pi_1(\tSig),G)/G)^\tau$ corresponds precisely to the component
$\cM_r^\irr$ of $\Hom^\irr_\tau(\pi_1(\Sig),PG)/PG$ which consists of flat
$PG$-bundles over $\Sig$ of topological type $r\in Z(G)/2Z(G)$.
In particular, $\cN^\irr_0$ corresponds to the component $\cM^\irr_0$ of
topologically trivial flat $PG$-bundles over $\Sig$.

The results in subsection shows part~(3) of Theorem~\ref{thm:cover}.

\section{Comparison of representation variety and gauge theoretical
constructions}\label{sec:geom}

Suppose $M$ is a compact non-orientable manifold, $\pi\colon\tM\to M$ is
the oriented cover, and $\tau\colon\tM\to\tM$ is the non-trivial deck
transformation.
In subsection~\ref{sec:DC}, we considered the natural lift of $\tau$ on
$\tP^\bC=\pi^*P^\bC$, where $P^\bC$ is a principal $G$-bundle over $M$.
Such a lift, still denoted by $\tau$, is a $G$-bundle map satisfying
$\tau^2=\id_{\tP^\bC}$ and induces involutions on the space $\cA(\tP^\bC)$
of connections on $\tP^\bC$ and various moduli spaces.
Moduli spaces associated to $P^\bC\to M$ are then related to the
$\tau$-invariant parts of those associated to $\tP^\bC\to\tM$
(cf.~Theorem~\ref{thm:Hitchin}, especially part~3).
This can also be seen in the language of representation varieties
(cf.~Lemma~\ref{lem:pro1}, Proposition~\ref{pro:pro2} on $\cN^\good_0$
and Corollary~\ref{cor:localdiff}).
To provide a geometric interpretation of the rest of the results in
subsections~\ref{sec:pro1&2} and \ref{sec:pro3} on $\cN^\good_r$ or
$\cN^\irr_r$ when $r\ne0$, we will need to generalize the setting in
gauge theory.

Suppose $Q\to\tilde M$ is a principal $G$-bundle and the non-trivial deck
transformation $\tau$ on $\tilde M$ is lifted to a bundle map $\tau_Q$
on $Q$, which is not necessarily an involution.
Let $A$ be an irreducible connection on $Q$ that is invariant under $\tau_Q$
up to a gauge transformation, i.e., $\tau_Q^*A=\vph^*A$ for $\vph\in\cG(Q)$.
Since $(\tau_Q\circ\vph^{-1})^2$ is a gauge transformation on $Q$ which fixes
$A$, it is in the center $Z(G)$.
So by modifying $\tau_Q$ with a gauge transformation $\vph$, we can assume
that $\tau_Q$ satisfies $\tau_Q^2=z\in Z(G)$.
In this way, although $\tau_Q$ is not strictly an involution, it is so up to
a gauge transformation, the right action of $z$ on $Q$.
Since $\vph$ and hence $\tau_Q$ can be adjusted by an element in $Z(G)$,
$z=\tau_Q^2$ is well defined modulo $2Z(G)$.
If $z=t^2\in2Z(G)$ ($t\in Z(G)$), then $z$ can be absorbed in $\tau_Q$ by
a redefinition such that $\tau_Q$ is an honest involution, and we are back
to the situation before.
In the general case when $\tau_Q^2=z\in Z(G)$ is not the identity element,
since $Z(G)$ acts trivially on the connections as gauge transformations,
the action $\tau_Q^*\colon\cA(Q)\to\cA(Q)$ of $\tau_Q$ on connections is
still an honest involution.
So we can define the invariant subspace $\cA(Q)^{\tau_Q}$ and much of the
analysis in subsections~\ref{sec:DC} and \ref{sec:Hitchin} applies.

We now consider flat connections and relate this generalized setting to our
results on representation varieties.
Choose base points $x_0\in M$ and $\tilde x_0\in\pi^{-1}(x_0)\subset\tM$,
and let $\Gam=\pi_1(M,x_0)$, $\tGam=\pi_1(\tM,\tilde x_0)$.
We fix an element $c\in\Gam\setminus\tGam$.

\begin{pro}\label{pro:general-tau}
For any $z\in Z(G)$, there is a 1-1 correspondence between the following
two sets:\\
(1) isomorphism classes of pairs $(Q,A)$, where $Q\to\tM$ is a principal
$G$-bundle with a $G$-bundle map $\tau_Q$ lifting the deck transformation
$\tau$ on $\tM$ satisfying $\tau_Q^2=z$, $A$ is a $\tau_Q$-invariant
flat connection on $Q$\\
and\\
(2)
equivalence classes of pairs $(\phi,g)$ under the diagonal adjoint action
of $G$, where $\phi\in\Hom(\tGam,G)$ and $g\in G$ satisfy
$\phi\circ\Ad_c=\Ad_g\circ\phi$ and $g^2\phi(c^2)^{-1}=z$.
\end{pro}

\begin{proof}
Given a bundle $Q$ and a $\tau_Q$-invariant flat connection $A$,
let $T_\al\colon Q_{\al(0)}\to Q_{\al(1)}$ be the parallel transport
along a path $\al\colon[0,1]\to\tM$.
$\tau_Q$-invariance of the connection implies
$\tau_Q\circ T_\al=T_{\tau\circ\al}\circ\tau_Q$ for any path $\al$.
Let $\gam$ be a path in $\tM$ from $\tilde x_0$ to $\tau(\tilde x_0)$
so that $[\pi\circ\gam]=c$.
Choose $q_0\in Q_{\tilde x_0}$ and let $g\in G$ be defined by
$T_\gam q_0=\tau_Q(q_0)g^{-1}$.
On the other hand, define $\phi\in\Hom(\tGam,G)$ by
$T_\al q_0=q_0\phi(a)^{-1}$ for any $a\in\tGam$, where $\al$ is a loop
in $\tM$ based at $\tilde x_0$ such that $[\al]=a$.
To check the conditions on $(\phi,g)$, we note that
$\tau_Q(T_\al q_0)=\tau_Q(q_0)\phi(a)^{-1}$ and
$$T_{\tau\circ\al}\tau_Q(q_0)
=T_\gam\circ T_{\gam\cdot(\tau\circ\al)\cdot\gam^{-1}}(q_0g)
=(T_\gam q_0)\phi(\Ad_ca)g=\tau_Q(q_0)\Ad_g^{-1}\phi(\Ad_ca).$$
So $\tau_Q$-invariance implies $\phi(\Ad_ca)=\Ad_g\phi(a)$ for all
$a\in\tGam$.
Similar calculations give
$\tau_Q(T_\gam q_0)=\tau_Q(\tau_Q(q_0)g^{-1})=q_0zg^{-1}$ and
$T_{\tau\circ\gam}(\tau_Qq_0)=T_{\gam\cdot(\tau\circ\gam)}(q_0g)$\\
$=q_0\phi(c^2)^{-1}g$ which imply $g^2\phi(c^2)^{-1}=z$.
If another point $q'_0=q_0h\in Q_{\tilde x_0}$ is chosen (where $h\in G$),
then the resulting pair is $(\phi',g')=(\Ad_{h^{-1}}\circ\phi,\Ad_{h^{-1}}g)$.

Conversely, given a pair $(\phi,g)$ satisfying the conditions, we want to
construct a bundle $Q$ together with a lifting $\tau_Q$ of $\tau$ such that
$\tau_Q^2=z$ and a $\tau_Q$-invariant flat connection on $Q$.
Let $\hat M$ be the universal covering space of $\tM$ (and of $M$).
Then $\tGam$ and $\Gam$ act on $\hat M$, and $\tM=\hat M/\tGam$,
$M=\hat M/\Gam$.
Let $Q=\hat M\times_{\tGam}G$, that is, points in $Q$ are equivalence classes
$[(x,h)]$, where $x\in\hat M$ and $h\in G$, and $(xa,h)\sim(x,\phi(a)h)$ for
any $a\in\tGam$.
Let $\tau_Q\colon Q\to Q$ be defined by
$\tau_Q\colon[(x,h)]\mapsto[(xc^{-1},gh)]$.
To check that $\tau_Q$ is well-defined, we note that for any $a\in\tGam$,
$(xac^{-1},gh)\sim(xc^{-1},\phi(\Ad_ca)gh)=(xc^{-1},g\phi(a)h)$.
Clearly, $\tau_Q$ commutes with the right $G$-action on $Q$.
Furthermore, $\tau_Q^2=z$ because $\tau_Q^2\colon[(x,h)]\mapsto
[(xc^{-2},g^2h)]=[(x,\phi(c^{-2})g^2h)]=[(x,h)]z$.
It is easy to see that the trivial connection on $\hat M\times G$ is
$\tGam$-invariant and descends to a flat connection on $Q$.
The latter is invariant under $\tau_Q$ since the trivial connection on
$\hat M\times G$ is invariant under $(x,h)\mapsto(xc^{-1},gh)$.
Moreover, this connection induces the pair $(\phi,g)$.
\end{proof}

\begin{rem}
We explain the gauge theoretic perspective of the results in
subsections~\ref{sec:pro1&2} and \ref{sec:pro3} using the correspondence
in Proposition~\ref{pro:general-tau}.\\
1. As we noted, the $\tau$ is lifted to a $G$-bundle map $\tau_Q$ on
$Q\to\tM$ such that $\tau_Q^2=z\in Z(G)$, then $z$ is determined up to $2Z(G)$.
Likewise, $z=g^2\phi(c^2)^{-1}$ is determined also modulo $2Z(G)$ by
$[\phi]\in(\Hom^\good(\tGam,G)/G)^\tau$ (Proposition~\ref{pro:pro1}).
If $\tau_Q^2=t^2$ for some $t\in Z(G)$, then $\tau_Q$ can be redefined as
$\tau'_Q=\tau_Q t^{-1}$ so that $(\tau'_Q)^2=\id_Q$.
We then have a $G$-bundle $Q/\tau'_Q\to M$ over the non-orientable manifold
$M$ whose pull-back of to $\tM$ is $Q$.
If a flat connection is invariant under $\tau_Q$, it is also invariant under
$\tau'_Q$ and hence descends to a flat connection on $Q/\tau'_Q$.
This is the situation in Lemma~\ref{lem:pro1} and Proposition~\ref{pro:pro2}
(where $Q/\tau'_Q$ was $P^\bC$).
In fact, from these results, we see that $[z]\in Z(G)/2Z(G)$ is the obstruction
to the existence of a flat $G$-bundle on $M$ whose pull-back to $\tM$ is $Q$.\\
2. In general, $\tau_Q^2\ne\id_Q$ and the quotient of $Q$ by the subgroup
generated by $\tau_Q$ is a bundle over $M$ with a fibre smaller than $G$.
However, the $PG$-bundle $\bar Q:=Q/Z(G)$ over $\tM$ does have an honest
involution $\tau_{\bar Q}$.
So $\bar Q$ descends to a $PG$-bundle $\bar Q/\tau_{\bar Q}$ over $M$.
Moreover, a $\tau_Q$-invariant flat connection on $Q$ descends to a
$\tau_{\bar Q}$-invariant flat connection on $\bar Q$ and hence to a flat
$PG$-connection on $\bar Q/\tau_{\bar Q}$.
The bundle $\bar Q/\tau_{\bar Q}\to M$ is usually non-trivial as its structure
group can not be lifted to $G$.
(Otherwise, $Q$ would be its pull-back to $\tM$ and would admit a lift
$\tau_Q$ of $\tau$ so that $\tau_Q^2=\id_Q$.)
Proposition~\ref{pro:pro3} shows that when $G$ is simply connected and
when $M=\Sig$ is a non-orientable surface, the topological type, i.e.,
the obstruction to lifting the $PG$-bundle $\bar Q/\tau_{\bar Q}$ to a
$G$-bundle over $M$ is precisely $[z]\in Z(G)/2Z(G)$.
\end{rem}

\begin{rem}
1. We can use $\tilde x_1=\tau(\tilde x_0)$ as an another base point of the
fundamental group of $\tM$ so that $\tilde x_0$ and $\tilde x_1$ play
symmetric roles.
The image of $\pi_1(\tSig,\tilde x_1)$ under $\pi_*$ can be identified with
$\tGam\subset\Gam$.
The isomorphism $\tau_*\colon\tGam\to\pi_1(\tSig,\tilde x_1)\cong\tGam$ is
then $a\mapsto\Ad_c^{-1}a$.
Having chosen $q_0\in Q_{\tilde x_0}$, let $q_1=\tau_Q(q_0)\in Q_{\tilde x_1}$
and define $\phi_1\colon\pi_1(\tSig,\tilde x_1)\to G$ by
$T_\al q_1=q_1\phi_1([\al])^{-1}$, where $\al$ is a loop in $\tSig$ based at
$\tilde x_1$.
Using the identity $\tau_Q\circ T_{\tau\circ\al}=T_\al\circ\tau_Q$, we obtain
$\phi_1([\al])=\phi([\tau\circ\al])$.
Since $\tau_Q^2=z$, we also have the identity
$T_\gam z=\tau_Q\circ T_{\tau\circ\gam}\circ\tau_Q$.
So upon the identification of $Q_{\tilde x_0}$ and
$Q_{\tilde x_1}$ by $\tau_Q$, the parallel transports along $\gam$ and
$\tau\circ\gam$ differ by $z$.\\
2. When $M=\Sig$ is a non-orientable surface, the approach of double base
points was taken in \cite{Ho,HL2}.
Consider for example the case $M=\Sig^\ell_1$.
Let $\al_i$, $\beta_i$ ($1\le i\le\ell$) be loops in the oriented cover $\tSig$
based at $\tilde x_0$ and let $\gam$ be a path in from $\tilde x_0$ to
$\tilde x_1$
so that $[\pi\circ\al_i]=a_i$, $[\pi\circ\beta_i]=b_i$, $[\pi\circ\gam]=c$.
Then an element in $\cN_r$ ($r=[z]\in Z(G)/2Z(G)$) can be represented by
$(A_i,B_i,C;A'_i,B'_i,C')\in G^{4\ell+2}$ satisfying $A'_i=A_i$, $B'_i=B_i$,
$C'=Cz$, where $A_i,B_i,C,A'_i,B'_i,C'$ are the holonomies along the loops
or paths $\al_i,\beta_i,\gam,\tau\circ\al_i,\tau\circ\beta_i,\tau\circ\gam$,
($1\le i\le\ell$), respectively.
By the above discussion, we have the pattern
$A_i=\phi([\al_i])=\phi_1([\tau\circ\al_i])=A'_i$,
$B_i=\phi([\beta_i])=\phi_1([\tau\circ\beta_i])=B'_i$, $(1\le i\le\ell)$,
$C'=Cz$ as in \cite{Ho,HL2}.
\end{rem}

\section*{Acknowledgments}

We would like to thank U.~Bruzzo, W.~Goldman and E.~Xia for useful discussions.

The research of N.H. was supported by grant number 99-2115-M-007-008-MY3 from the National Science Council of Taiwan.
The research of G.W. was supported by grant number R-146-000-200-112 from the National University of Singapore.
The research of S.W. was partially supported by RGC grant HKU705612P (Hong Kong) and by MOST grant 105-2115-M-007-001-MY2 (Taiwan).

\end{document}